\documentclass[12pt, a4paper]{amsart}

\usepackage{amsmath, amsthm, amssymb}
\usepackage{graphicx}
\usepackage[margin = 2.5cm]{geometry}

\theoremstyle{definition}
\newtheorem{definition}{Definition}
\newtheorem{remark}{Remark}
\newtheorem{example}{Example}
\theoremstyle{theorem}
\newtheorem{lemma}{\bf Lemma}
\newtheorem{proposition}{\bf Proposition}
\newtheorem{theorem}{\bf Theorem}

\usepackage{algorithm, algorithmic}

\makeatletter
\renewcommand{\p@algorithm}{\arabic{algorithm}\expandafter\@gobble}
\makeatother 
\newcommand{\PARAMETERS}{\item[\textbf{Parameters:}]}

\newcounter{step}[algorithm]
\newcommand\STEP[2][\(\triangleright\)]{%
	\refstepcounter{step}
	\vskip 0.25\baselineskip
	\item[]\hskip -\algorithmicindent #1 \textbf{Step \arabic{step}}%
	\ifthenelse{\equal{\unexpanded{#2}}{}}{}{ (\texttt{#2})}%
	\textbf{.}%
}
\newenvironment{algo}{\algo}{}
\def\algo#1\end{%
	\noindent\fbox{%
	\begin{minipage}[b]{\dimexpr\columnwidth-\algorithmicindent\relax}
	\begin{algorithmic}
	#1
	\end{algorithmic}
	\end{minipage}
	}%
\end}

\usepackage[T1]{fontenc}
\makeatletter

	\renewcommand{\@secnumfont}{\bfseries}
    \def\section{\@startsection{section}{1}%
    \z@{.7\linespacing\@plus\linespacing}{.5\linespacing}%
    {\normalfont\bfseries\scshape \centering}}
\makeatother

\usepackage{isomath}
\newcommand{\x}{x}
\newcommand{\y}{y}
\newcommand{\g}{g}
\newcommand{\del}{\delta}
\newcommand{\vv}{v}
\newcommand{\uu}{u}
\newcommand{\p}{p}
\newcommand{\q}{q}
\newcommand{\h}{h}
\newcommand{\dd}{d}
\newcommand{\z}{z}

\usepackage{mathrsfs} 
\usepackage{color}

\begin{document}

\title[Nonsmooth trust-region algorithm]{Nonsmooth trust-region algorithm with applications to robust stability of uncertain systems}

\author[P. Apkarian, D. Noll, L. Ravanbod]{Pierre Apkarian$^\dag$, Dominikus Noll$^*$, Laleh Ravanbod$^*$}
\thanks{$^\dag$Control System Department, ONERA, Toulouse, France}
\thanks{$^*$Institut de Math\'ematiques, Universit\'e de Toulouse, France}
\date{}

\maketitle

\begin{abstract}
We propose a bundle trust-region algorithm
to minimize locally Lipschitz functions which are potentially nonsmooth
and nonconvex. We prove global convergence of our method and show by way of an example
that the classical convergence argument in trust-region methods based
on the Cauchy point fails in the nonsmooth setting.
Our method is tested experimentally on three problems in automatic control.

\end{abstract}

\maketitle

\section{Introduction}
We consider optimization problems of the form
\begin{eqnarray}
\label{program}
\begin{array}{ll}
\mbox{minimize} & f(\x)\\
\mbox{subject to} & \x\in C\\
\end{array}
\end{eqnarray}
where $f:\mathbb R^n\to \mathbb R$ is  locally Lipschitz, but possibly nonsmooth and nonconvex, and where $C$
is a simply structured closed convex constraint set.  
We develop a bundle trust-region algorithm for (\ref{program}),  which uses 
nonconvex cutting planes in tandem with
a suitable trust-region management  to assure global convergence. 
The trust-region management is to be considered as an alternative to proximity control, which is the usual policy
in bundle methods. 
Trust-regions allow a tighter control on the step-size, and give a larger choice of norms,
whereas bundling is fused on the use of the Euclidean norm.   Our experimental part demonstrates how
these features may be exploited algorithmically.

Algorithms where bundle and trust-region elements  are combined are rather sparse in the literature.
For convex objectives Ruszcy\'nski \cite{rud} presents a bundle  trust-region method, 
which  can be extended to  composite convex functions. 
An early contribution where bundling and trust-regions are combined is
\cite{thesis_schramm,schramm}, and this is also used in versions of the BT-code \cite{zowe}.
Fuduli {\em et al.} \cite{fuduli}
use DC-functions to form a non-standard trust-region, which they also use in tandem with cutting planes.
A feature which these methods share  with nonconvex bundle methods like
 Sagastiz\'abel and Hare \cite{hare1,saga} or \cite{NPR08}  is that the objective
 is approximated by a simply structured, often polyhedral, working model, 
 which is updated iteratively by adding cutting planes at unsuccessful trial steps.
 Our main Theorem \ref{theorem1} analyses the interaction of this mechanism with the trust-region management, and assures global
 convergence  under realistic hypotheses.
 
The trust-region strategy is well-understood in smooth optimization, where  global convergence is proved
by exploiting properties of the Cauchy point,
as pioneered in Powell \cite{powell}. For the present work it is therefore of the essence 
to realize that the Cauchy point
fails  in the nonsmooth setting. This happens  even for polyhedral convex functions, the simplest possible
case, as we demonstrate by way of a counterexample.  This explains why the convergence proof has to be organized along different lines.

The question is then  whether there are
more restrictive classes of nonsmooth functions, where the Cauchy point can be
salvaged. In response we show that the classical trust-region strategy with Cauchy point
is still valid for upper $C^1$-functions, and at least partially, for functions having a strict
standard model.  It turns out that several problems in control and in contact mechanics are
in this class, which justifies the disquisition.  Nonetheless, the class of functions where the Cauchy point 
works remains exceptional in the nonsmooth framework, which is corroborated by the fact that
it does not  include nonsmooth convex functions.

A strong incentive for the present work comes indeed from
applications in automatic control. 
In the experimental part we will
apply our novel bundle trust-region method to compute locally optimal
solutions to three NP-hard problems in the theory of systems with uncertain parameters.
This includes (i) computing the worst-case
$H_\infty$-norm of a system  over a given uncertain parameter range, 
(ii)  checking robust stability of an uncertain system over a given parameter range,
and (iii) computing the distance to instability
of a nominally stable system with uncertain parameters.  In these applications
the versatility of the bundle trust-region approach with regard to the choice of the
norm is exploited.

Nonsmooth trust-region methods which do {\em not} include the possibility of bundling are more common,
see for instance
Dennis {\em et al.} \cite{dennis}, where the authors present an axiomatic approach, and
\cite[Chap. 11]{conn}, where that idea is further  expanded. 
A recent trust-region method for DC-functions is \cite{DC}.

The structure of the paper is as follows. The algorithm
is developed in section \ref{sect_algo}, and its global convergence is proved in section \ref{sect_convergence}.
Applications of the model approach are discussed in section \ref{sect_applications}, where
we also discuss failure of the Cauchy point.  Numerical experiments with
three problems in  automatic control are presented in section \ref{sect_experiments}.

\section*{Notation}
For nonsmooth optimization we follow
\cite{Cla83}. 
The Clarke directional derivative of $f$ is $f^\circ(\x,\dd)$, its Clarke subdifferential $\partial f(\x)$. For a function $\phi$ of two variables
$\partial_1\phi$ denotes the Clarke subdifferential with respect to the first variable. For symmetric matrices $M\preceq 0$
means negative semidefinite. For linear system theory see
\cite{ZDG96}.

\section{Presentation of the algorithm}
\label{sect_algo}
In this chapter we derive our 
trust-region algorithm to solve program (\ref{program}) and discuss its building blocks.

\subsection{Working model}
We start by explaining how a local approximation of $f$
in the neighborhood of the current serious iterate $\x$, called the {\em working model}  of $f$, is generated iteratively.
We recall the notion of a first-order model of $f$ introduced in \cite{NPR08}.

\begin{definition}
\label{def1}
A function
$\phi:\mathbb R^n\times \mathbb R^n\to \mathbb R$ is called a {\em first-order model} of $f$ on a set $\Omega$ if
$\phi(\cdot,\x)$ is convex for every $\x\in \Omega$, and the following properties are satisfied:
\begin{enumerate}
\item[$(M_1)$] $\phi(\x,\x) = f(\x)$, and $\partial_1 \phi(\x,\x)\subset \partial f(\x)$.
\item[$(M_2)$]  If $\y_k \to \x$, then there exist $\epsilon_k\to 0^+$ such that
$f(\y_k)\leq \phi(\y_k,\x) + \epsilon_k \|\y_k-\x\|$. 
\item[$(M_3)$] If $\x_k\to \x$, $\y_k\to \y$, then
$\limsup_{k\to\infty} \phi(\y_k,\x_k) \leq \phi(\y,\x)$. \hfill $\square$
\end{enumerate}
\end{definition}

We may think of $\phi(\cdot,\x)$ as  a  non-smooth first-order Taylor expansion of $f$ at $\x$.
Every locally Lipschitz function has indeed a first-order model $\phi^\sharp$, which we call the {\em standard model},
defined as
\[
\phi^\sharp(\y,\x) = f(\x) + f^\circ(\x,\y-\x).
\]
Here $f^\circ(\x,{\dd})$ is the Clarke directional derivative of $f$ at $\x$ in direction ${\dd}$.
Following \cite{NPR08}, a first-order model $\phi(\cdot,\x)$ is called {\em strict} at $\x\in\Omega$ if the following
strict version of $(M_2)$ is satisfied:
\begin{enumerate}
\item[$(\widetilde{M}_2)$]  Whenever $\y_k \to \x$, $\x_k\to \x$, there exist $\epsilon_k\to 0^+$ such that
$f(\y_k)\leq \phi(\y_k,\x_k) + \epsilon_k \|\y_k-\x_k\|$. 
\end{enumerate}

\begin{remark}
Axiom $(M_2)$ corresponds to the one-sided Taylor type estimate
$f(\y) \leq \phi(\y,\x) + {\rm o}(\|\y-\x\|)$ as $\y\to \x$. In contrast, axiom $(\widetilde{M}_2)$
means $f(\y) \leq \phi(\y,\x) + {\rm o}(\|\y-\x\|)$ as $\|\y-\x\| \to 0$ uniformly
on bounded sets. This is analogous to the difference between differentiability and strict differentiability,
hence the nomenclature of a strict model.
\end{remark}

\begin{remark}
Note that the standard model $\phi^\sharp$ of $f$ is not always strict \cite{Noll2010}. A strict first-order model $\phi$ is for instance
obtained for composite  functions $f = h\circ F$ with $h$ convex and $F$ of class $C^1$, if one defines
\[
\phi(\y,\x) = h\left( F(\x) + F'(\x)(\y-\x) \right),
\] 
where $F'(\x)$ is the differential of the mapping $F$ at $\x$.  The use of a natural
model of this form covers for instance approaches
like Powell \cite{powell},  or Ruszczy\'nski \cite{rud}, where composite functions are
discussed.

Observe that every convex $f$  is its own strict model $\phi(\y,\x)=f(\y)$ in the sense of definition
\ref{def1}. As a consequence, our algorithmic framework contains the convex cutting plane trust-region method
\cite{rud} as a special case.
\end{remark}

\begin{remark}
It follows from the previous remark that a function $f$ may have several first-order models. 
Every model $\phi$ leads to a different algorithm for (\ref{program}).
\end{remark}

We continue to consider $\x$ as the current serious iterate of our algorithm to be designed,
and we consider $\z$,  a trial point near $\x$, which is a candidate to become the next serious iterate $\x^+$.
The way trial points are generated will be explained in Section \ref{sect_tangent}.

\begin{definition}
Let $\x$ be the current serious iterate and $\z$ a trial step.
Let $\g$ be a subgradient of $\phi(\cdot,\x)$ at $\z$, for short,
$\g\in \partial_1\phi(\z,\x)$. Then the affine function
$m(\cdot,\x) = \phi(\z,\x) + \g^\top (\cdot - \z)$ is called a {\em cutting plane} of $f$ at serious iterate $\x$
and trial step $\z$. \hfill $\square$
\end{definition}

We may always represent a cutting plane at serious iterate $\x$ in the form
\[
m(\cdot,\x) = a + \g^\top (\cdot - x),
\]
where $a = m(\x,\x) = \phi(\z,\x) + \g^\top (\x-\z)\leq f(\x)$ and $\g\in \partial_1\phi(\z,\x)$. 
We say that the pair $(a,\g)$ represents the cutting plane $m(\cdot,\x)$. 

We also allow cutting planes $m_0(\cdot,\x)$  at serious iterate $\x$ with trial step 
$\z = \x$.
We refer to these
as {\em exactness planes} of $f$ at serious iterate $\x$, because $m_0(\x,\x) = f(\x)$.
Every $(a,\g)$ representing an exactness plane is of the form $(f(\x),\g_0)$ with $\g_0\in \partial f(\x)$.

\begin{remark}
For the standard model $\phi^\sharp$ a cutting plane for trial step $\z$ at serious iterate $\x$
has the very specific form $m^\sharp(\cdot,\x) = f(\x) + \g_z^\top (\cdot-\x)$, where
$\g_z\in \partial f(\x)$ attains the maximum $f^\circ(\x,\z-\x)=\g_z^\top (\z-\x)$. Here every cutting plane $m^\sharp(\cdot,\x)$ is
also an exactness plane, a fact which will no longer be true for other  models. If $f$ is strictly differentiable at $\x$, then there is only one cutting plane
$m^\sharp(\cdot,\x)=f(\x)+\nabla f(\x)^\top (\cdot-\x)$, the first-order Taylor polynomial.
\end{remark}

\begin{definition}
\label{def_working}
Let $\mathcal G_k$ be a set of pairs $(a,\g)$ all representing cutting planes of $f$ at trial steps
around the serious iterate $\x$.
Suppose $\mathcal G_k$ contains at least one exactness plane at  $\x$.
Then 
$\phi_k(\cdot,\x) = \max_{(a,\g)\in \mathcal G_k} a + \g^\top (\cdot - \x)$
is called a {\em working model} of $f$ at $\x$.
\hfill $\square$
\end{definition} 
 
 \begin{remark}
 We index working models $\phi_k$ by the inner loop counter $k$ to highlight
 that they are updated in the inner loop by adding 
 tangent planes of the ideal model $\phi$ at the null steps $\y^k$.
 
 Usually the $\phi_k$ are rough polyhedral approximation of $\phi$, 
 but we do not exclude cases where
 the $\phi_k$ are generated by infinite sets $\mathcal G_k$. 
 This is for instance the case in the spectral bundle method  \cite{helmberg1,helmberg2,helmberg3}, see also
\cite{ANP}, which we discuss this in 
\ref{spectral}.
 \end{remark}
 
 \begin{remark}
 Note that even the choice $\phi_k=\phi$ is allowed in definition \ref{def_working} and in algorithm \ref{algo1}. This corresponds to
 $\mathcal G=\{(a,g): g\in \partial f(\z), a=\phi(\z,\x)+\g^\top(\x-\z)\}$, which is the largest possible
 set of cuts, or the set of all cuts obtained from $\phi$. We discuss this case in section \ref{full}. If  
 $\phi^\sharp$ is used, then the corresponding working models are denoted $\phi^\sharp_k$. Their case is analyzed in section \ref{sect_standard}.
 \end{remark}
 
 The properties of a working model may be summarized as follows
 
 \begin{proposition}
 Let $\phi_k(\cdot,\x)$ be a working model of $f$ at $\x$ built from $\mathcal G_k$ and based on the ideal model $\phi$. Then
 \begin{enumerate}
 \item[(i)] $\phi_k(\cdot,\x) \leq \phi(\cdot,\x)$.
 \item[(ii)] $\phi_k(\x,\x) = \phi(\x,\x)=f(\x)$.
 \item[(iii)] $\partial_1 \phi_k(\x,\x) \subset \partial_1 \phi(\x,\x)\subset \partial f(\x)$.
 \item[(iv)] If $(a,g)\in \mathcal G_k$ contributes to $\phi_k$ and stems from the trial step $\z$ at serious
 iterate $\x$, then $\phi_k(\z,\x)=\phi(\z,\x)$.
 \end{enumerate}
 \end{proposition}
 
 \begin{proof}
 By construction $\phi_k$ is a maximum of affine minorants of $\phi$, which proves (i). Since at least
 one plane in $\mathcal G_k$ is of the form $m_0(\cdot,\x) = \phi(\x,\x) + \g^\top(\cdot-\x)$ with $\g\in \partial_1\phi(\x,\x)$, we have
 $\phi_1(\x,\x) \geq m_0(\x,\x) = \phi(\x,\x) = f(\x)$, which proves (ii). To prove (iii), observe that since
 $\phi_k(\cdot,\x)$ is convex, every $\g\in \partial_1\phi_k(\x,\x)$ gives an affine minorant
 $m(\cdot,\x) = \phi_k(\x,\x) + \g^\top (\cdot-\x)$ of $\phi_k(\cdot,\x)$. Then $m(\cdot,\x)\leq \phi(\cdot,\x)$
 with equality at $\x$. By convexity  $\g\in \partial_1 \phi(\x,\x)$, and by axiom $(M_1)$ we have
 $\g\in \partial f(\x)$. As for (iv), observe that every cutting plane $m(\cdot,\x)$ at $\z$ satisfies $m(\z,\x)=\phi(\z,\x)$,
 hence also $\phi_k(\z,\x) = \phi(\z,\x)$. 
 \end{proof}
 
\subsection{Tangent program}
\label{sect_tangent}
In this section we discuss how trial steps  are generated.
Given the current working model $\phi_k(\cdot,\x)
= \max\{ a + \g^\top(\cdot - \x): (a,\g)\in \mathcal G_k\}$, and the current trust-region radius $R_k$, the {\em tangent program}
is the following convex optimization problem
\begin{eqnarray}
\label{tangent}
\begin{array}{ll}
\mbox{minimize} & \phi_k(\y,\x)\\
\mbox{subject to} & \y\in C \\
&\|\y-\x\| \leq R_k
\end{array}
\end{eqnarray}
where $\|\cdot\|$ could be any norm on $\mathbb R^n$. 
Let $\y^k$ be an optimal solution of (\ref{tangent}). 
By the necessary optimality condition there exists a subgradient
$\g_k\in \partial \left(\phi_k(\cdot,\x)+i_C  \right)(\y^k)$ and a vector $\vv_k$
in the normal cone to $B(\x,R_k)$ at $\y^k\in B(\x,R_k)$ such that $0=\g_k+\vv_k$,
where $i_C$ is the indicator function of $C$.
We call $\g_k$ the {\em aggregate subgradient} at $\y^k$. This terminology stems
from the classical bundle method, when a polyhedral working
model is used, see Ruszczy\'nski \cite{rud},  Kiwiel \cite{kiwiel}.

Solutions $\y^k$ of (\ref{tangent}) are candidates to become the
next serious iterate $\x^+$. For practical reasons
we now enlarge the set of possible candidates.
Fix $0 < \theta \ll 1$ and $M\geq 1$, then every
$\z^k\in C \cap B(\x,M\|x-\y^k\|)$ satisfying
\begin{eqnarray}
\label{trial}
f(\x) - \phi_k(\z^k,\x) \geq \theta \left( f(\x)-\phi_k(\y^k,\x) \right)
\end{eqnarray}
is called a {\em trial step}. Note that $\y^k$ itself is of course a trial step, 
because $f(\x) \geq \phi_k(\y^k,\x)$ by the definition of the tangent program.
But due to $\theta\in (0,1)$,  there exists an
entire  neighborhood
$U$ of $\y^k$ such that every $\z^k\in U \cap C$ is a trial step.

\begin{remark}
The role of $\y^k$ here is not unlike 
that of the Cauchy point in classical trust-region  methods. Suppose we use a standard working model $\phi_k^\sharp$ and
$f$ is strictly differentiable at $\x$. Then 
$\phi_k^\sharp(\cdot,\x) = \phi^\sharp(\cdot,\x)=f(\x) + \nabla f(\x)^\top (\cdot-\x)$. In the unconstrained case
$C=\mathbb R^n$ the solution $\y^k$ has then the explicit form $\y^k=\x-R_k \frac{\nabla f(\x)}{\|\nabla f(\x)\|}$, 
which is indeed the Cauchy point as
considered in \cite{sartenaer}, see also \cite[(5.108)]{rud}. Condition (\ref{trial}) then takes the familiar  form
$f(\x)-\phi_k^\sharp(\z^k,\x) \geq \sigma \|\nabla f(\x)\| R_k$, 
see \cite[(5.110)]{rud}.
\end{remark}

\subsection{Acceptance test}
In order to decide whether a trial step
$\z^k$ will become the next serious iterate $\x^+$, we compute the test quotient
\begin{eqnarray}
\label{rho}
\rho_k
=
\frac{f(\x) - f(\z^k)}{f(\x)-\phi_k(\z^k,\x)},
\end{eqnarray}
which compares as usual actual progress and model predicted progress.
For a fixed parameter $0 < \gamma < 1$, the decision is as follows.
If $\rho_k \geq \gamma$, then the trial step $\z^k$ is accepted as the
new iterate $\x^+=\z^k$, and we call this a {\em serious step}. On the other hand, if $\rho_k < \gamma$, then $\z^k$
is rejected and referred to as a {\em null step}. In that case
we compute a cutting plane $m_k(\cdot,\x)$ at $\z^k$, and add it to the 
new set $\mathcal G_{k+1}$ in order to improve our working model. In other words, a pair $(a_{k},\g_{k})$
is added, where $\g_{k} \in \partial_1 \phi(\z^k,\x)$ and $a_{k}=\phi(\z^k,\x)+\g_{k}^\top(\x-\z^k)$.

\begin{remark}
Adding one cutting plane at the null step $\z^k$ is  mandatory, but
we may at leisure add several other tangent planes of $\phi(\cdot,\x)$ to further improve the working model.
A case of practical importance,  where the
$\phi_k$ are generated by infinite sets $\mathcal G_k$ of cuts, is presented in section \ref{spectral}.
\end{remark}

\begin{remark}
In most applications $\phi_k$ is a polyhedral convex function. If $C$ is also polyhedral, then it
is attractive to choose a polyhedral trust-region norm $\|\cdot\|$, because this makes (\ref{tangent}) a linear program.
\end{remark}

\begin{remark}
For polyhedral $\phi_k$ 
one can limit the size of the sets $\mathcal G_k$. 
Consider for simplicity $C=\mathbb R^n$, then the tangent program (\ref{tangent}) 
is $p=\min\{t: a_i + \g_i^\top (\y-\x) - t \leq 0, i=0,\dots,k, \|\z-\x\| \leq R_k\}$. Its  dual is
$d=\max\{\sum_{i=1}^k \lambda_i a_i - R_k\|\sum_{i=1}^k \lambda_i g_i\|: \lambda_i \geq 0, \sum_{i=1}^k \lambda_i=1\}$.
By Carath\'eodory's theorem we can select a subset $\{(a_0,g_0),\dots,(a_n,g_n)\}$ of $\mathcal G_k$ of size at most $n+1$
with the same convex hull as $\mathcal G_k$, so it is always possible to limit $|\mathcal G_k|\leq n+1$. This estimate
is pessimistic. An efficient but heuristic method is 
to remove from $\mathcal G_k$ a certain number of cuts which were not active at the last $\z^k$. 
In the bundle method with proximity control,  Kiwiel's aggregate subgradient
\cite{kiwiel} allows a rigorous theoretical limit of $|\mathcal G_k| \leq 3$, even though in practice one keeps
more cuts in the $\mathcal G_k$. It is not known whether
Kiwiel's argument can be extended to the trust-region case, and the only known bound is $n+1$,
see also \cite[Ch. 7.5]{rud} for a discussion. 
\end{remark}

\subsection{Nonsmooth solver}
\label{sect.solver}
	We are now ready to present
	our algorithm for program (\ref{program}).
	See Algorithm \ref{algo1} next page.

\begin{algorithm}
\caption{Nonsmooth trust-region method 
}\label{algo1}
\noindent\fbox{%
\begin{minipage}[b]{\dimexpr\textwidth-\algorithmicindent\relax}
\begin{algorithmic}
\PARAMETERS 
$0 <\gamma <\widetilde{\gamma} <1, 0 <\gamma <\Gamma \leq 1,  0 < \theta \ll 1$, $M \geq 1$. 
\STEP{Initialize outer loop} 
	Choose initial  iterate $\x^1 \in C$. Initialize memory trust-region radius as $R_1^\sharp >0$. Put $j =1$.
\STEP[$\diamond$]{Stopping test} 
	At outer loop counter $j$, stop if $\x^j$ is a critical point of (\ref{program}).
Otherwise, goto inner loop.
\STEP{Initialize inner loop} 
	Put inner loop counter $k =1$ and initialize trust-region radius  as $R_1=R_j^\sharp$. Build initial working model
	$ \phi_1(\cdot, \x^j)$ based on $\mathcal G_1$, where at least $(f(\x^j),\g_{0j})\in \mathcal G_1$ for some
	where $\g_{0j}\in\partial f(\x^j)$. Possibly enrich  $\mathcal G_1$ by recycling some of the planes from the previous
	serious step.
\STEP{Trial step generation} 
	At inner loop counter $k$ find solution $\y^k$ of the tangent program
\[
\begin{array}{ll}
\text{minimize} & \phi_k(\y, \x^j)\\
\text{subject to}& \y\in C\\
&\|\y-\x^j\|\leq R_k
\end{array}
\]
Then compute any  trial step $\z^k\in C \cap B(\x^j,M\|\x^j-\y^k\|)$ satisfying
$f(\x^j)-\phi_k(\z^k,\x^j) \geq \theta \left( f(\x^j)-\phi_k(\y^k,\x^j)  \right)$.
\vskip -0.25\baselineskip
\STEP[$\diamond$]{Acceptance test} 
	If 
	$$ \rho_k =\frac{f(\x^j)-f(\z^k)}{f(\x^j)-\phi_k(\z^k, \x^j)} \geqslant \gamma, $$
put $\x^{j+1} =\z^k$ (serious step), quit inner loop and goto step 8. Otherwise (null step), continue inner loop with step 6.
\STEP{Update working model} 
	Generate a cutting plane $m_k(\cdot, \x^j) =a_k +\g_k^\top(\cdot -\x^j)$ of $f$ at the null step
	$\z^k$ at counter $k$ belonging to the current serious step $\x^j$.
	Add $(a_k,\g_k)$ to $\mathcal G_{k+1}$.
	Possibly taper out  $\mathcal G_{k+1}$ by removing some of the
	older inactive planes in $\mathcal G_k$. Build $\phi_{k+1}$ based on $\mathcal G_{k+1}$.
\STEP[$\diamond$]{Update trust-region radius} 
	Compute secondary control parameter
	$$ \widetilde{\rho}_k =\frac{f(\x^j)-\phi(\z^k, \x^j)}{f(\x^j)-\phi_k(\z^k, \x^j)} $$
and put 
	$$ R_{k+1} =\begin{cases} R_k &\text{if } \widetilde{\rho}_k <\widetilde{\gamma},\\ \frac{1}{2}R_k &\text{if } \widetilde{\rho}_k \geqslant \widetilde{\gamma}.\end{cases} $$
Increase inner loop counter $k$ and loop back to step 4.
\STEP[$\diamond$]{Update  memory radius} 
	 Store new memory radius
	$$ R^{\sharp}_{j+1} =\begin{cases} R_{k} &\text{if } \rho_k < \Gamma,\\ {2}R_{k} &\text{if } \rho_k \geqslant \Gamma.\end{cases} $$
	
Increase outer loop counter $j$ and loop back to step 2.

\end{algorithmic}
\end{minipage}
}%
\end{algorithm}

\section{Convergence}
\label{sect_convergence}
In this chapter we analyze the convergence properties
of the main algorithm.

\subsection{Convergence of the inner loop}
In this section we  prove finiteness of the inner loop with counter $k$. Since the outer loop counter $j$ is fixed,
we simplify notation and write $\x =\x^j$ for the current serious iterate,
and $\x^+ = \x^{j+1}$ for the next serious iterate, which is the result
of the inner loop. 

\begin{lemma}
\label{lem_cauchy}
Let $\z^k$ be the trial point at inner loop instant $k$, associated with the solution $\y^k$ of  the tangent program, 
and let $\g_k$ be the aggregate subgradient at $\y^k$.
Then there 
exists $\sigma >0$ depending only on $\theta\in (0,1)$, $M$, and the norm $\|\cdot\|$, such that
\begin{eqnarray}
\label{cauchy1}
f(\x)-\phi_k(\z^k,\x) \geq \sigma \|g_k\| \|\x-\z^k\|.
\end{eqnarray}
\end{lemma}

\begin{proof}
Let $\|\cdot\|$ be the norm used in the trust-region tangent program, $|\cdot|$ the standard Euclidian norm.
Since $\y^k$ is an optimal solution of (\ref{tangent}), we have
$0=\g_k+\vv_k$, where $\g_k\in \partial\left( \phi_k(\cdot,\x)+i_C \right)(\y^k)$ and  $\vv_k$ a normal vector to $B(\x,R_k)$ at $\y^k$. 
By the subgradient inequality,
\[
\g_k^\top (\x-\y^k) \leq \phi_k(\x,\x)-\phi_k(\y^k,\x) = f(\x)-\phi_k(\y^k,\x).
\]
Now the angle between the vector $\y^k-\x$ and the normal $\vv_k$ to the $\|\cdot\|$-ball $B(\x,R_k)$ at $\y^k\in \partial B(\x,R_k)$
is strictly less than $90^\circ$. More precisely, there exists $\sigma'\in (0,1)$, depending only on the geometry
of the ball $B(0,1)$, such that $\cos \angle(\uu_k, \vv_k) \geq \sigma'$ for all such vectors $\uu_k,\vv_k$. But then
$\g_k^\top (\x-\y^k)=\vv_k^\top(\y^k-\x) \geq \sigma' |\vv_k| |\y^k-\x| \geq \sigma'' \|\vv_k\| \|\y^k-\x\|$ for some $\sigma''\in (0,1)$ still
depending only on the geometry of the norm $\|\cdot\|$. Invoking (\ref{trial}) for the trial point $\z^k$, and using $\|x-\z^k\| \leq M\|x-\y^k\|$, we get
(\ref{cauchy1}) with $\sigma = \sigma''\theta M^{-1}$.
\end{proof}

\begin{lemma}
Suppose the inner loop at $\x$ with trial point $\z^k$ at inner loop counter $k$
and solution $\y^k$ of the tangent program {\rm (\ref{tangent})}  turns infinitely,
and the trust-region radius $R_k$ stays bounded away from 0. Then $\x$ is a critical point of {\rm (\ref{program})}.
\end{lemma}

\begin{proof}
We have $\rho_k < \gamma$ for all $k$. Since $\liminf_{k\to\infty} R_k > 0$, 
and since the trust-region radius is only reduced when $\widetilde{\rho}_k \geq \widetilde{\gamma}$,
and is never increased during the inner loop, we conclude that there
exists $k_0$ such that $\widetilde{\rho}_k < \widetilde{\gamma}$ for all $k\geq k_0$,
and also
$R_k = R_{k_0}>0$ for all $k\geq k_0$.

As $\z^k$, $\y^k\in B(\x,R_{k_0})$, we can extract an infinite subsequence $k\in \mathcal K$ such that $\z^k\to \z$,
$\y^k\to \y$, $k\in \mathcal K$.
Since we are drawing cutting planes at $\z^k$, we have $\phi_k(\z^k,\x)=\phi(\z^k,\x)=m_k(\z^k,\x)$,
and then
$\phi_k(\z^k,\x)\to \phi(\z,\x)$. Therefore the numerator and denominator
in the quotient $\widetilde{\rho}_k$ both converge to
$\phi(\x,\x)-\phi(\z,\x)$, $k\in\mathcal K$. Since $\widetilde{\rho}_k < \widetilde{\gamma} < 1$  for all $k$, this could only
mean $\phi(\x,\x)-\phi(\z,\x)=0$. 

Now by condition (\ref{trial}) we have 
$$\phi(\x,\x)-\phi_k(\y^k,\x)\leq \theta^{-1}\left(\phi(\x,\x)-\phi_k(\z^k,\x)\right)\to 0,$$ hence
$\limsup_{k\in \mathcal K} \phi(\x,\x)-\phi_k(\y^k,\x) \leq 0$. On the other hand, $\phi_k(\y^k,\x) \leq \phi(\x,\x)$ 
since $\y^k$ solves the tangent program, hence $\phi_k(\y^k,\x)\to \phi(\x,\x)$, too.

By the necessary optimality condition for the tangent program
(\ref{tangent}) there exist $\g_k\in \partial_1\phi_k(\y^k,\x)$ and a normal vector $\vv_k$ to $C \cap B(\x,R_{k_0})$
at $\y^k$ such that
$0=\g_k+\vv_k$. By boundedness of the $\y^k$ and local boundedness of the subdifferential,
the sequence $g_k$ is bounded, and hence so is the sequence $\vv_k$. Passing to yet another subsequence $k\in \mathcal K'
\subset \mathcal K$, we may assume $\g_k\to \g$, $\vv_k\to\vv$, and by upper semi-continuity of the subdifferential,
$\g\in \partial_1\phi(\y,\x)$, and
$\vv$ is in the normal cone to $C \cap B(\x,R_{k_0})$ at $\y$.
Since $0=\g+\vv$, we deduce that $\y$ is a critical point of the optimization program
$\min\{\phi(\y,\x): \y\in C \cap B(\x,R_{k_0})\}$, and since this is a convex program, $\y$ is a minimum. But from the previous
argument we have seen that $\phi(\y,\x)=\phi(\x,\x)$, and since $\x$ is admissible for that program, it is also
a minimum. 
A simple convexity argument now shows that $\x$ is a minimum of
 (\ref{tangent}). 
\end{proof}

\begin{lemma}
\label{lemma2}
Suppose the inner loop at $\x$  with trial point $\z^k$  and solution $\y^k$ of the tangent program at inner
loop counter $k$  turns forever,  
and $\liminf_{k\to \infty} R_k=0$.
Then $\x$ is a critical point of {\rm (\ref{program})}.
\end{lemma}

\begin{proof}
This proof uses (\ref{cauchy1}) obtained in Lemma \ref{lem_cauchy}.
We are in the case where $\widetilde{\rho}_k \geq \widetilde{\gamma}$ for infinitely many
$k\in \mathcal N$. Since $R_k$ is never increased in the inner loop, we have  $R_k\to 0$.
Hence $\y^k,\z^k\to \x$ as $k\to\infty$. 

We claim that $\phi_k(\z^k,\x)\to f(\x)$. Indeed, we clearly have
$\limsup_{k\to\infty} \phi_k(\z^k,\x)\leq \limsup_{k\to\infty} \phi(\z^k,\x) =\lim_{k\to\infty} \phi(\z^k,\x)=f(\x)$.
On the other hand, the exactness plane $m_0(\cdot,\x)=f(\x) + g_0^\top (\cdot-\x)$ is an affine minorant
of $\phi_k(\cdot,\x)$ at all times $k$, hence
$f(\x) = \lim_{k\to\infty} m_0(\y^k,\x) \leq \liminf_{k\to\infty} \phi_k(\y^k,\x)$, and the two together show
$\phi_k(\z^k,\x)\to f(\x)$. 

By condition (\ref{cauchy1})  we have
$f(\x)-\phi_k(\z^k,\x) \geq \sigma \|\g_k\| \|\x-\z^k\|$, where $\g_k\in
\partial\left(  \phi_k(\cdot,\x)+i_C\right)(\y^k)$ is the aggregate subgradient. Now assume that $\|\g_k\|\geq \eta > 0$ for all
$k$. Then $f(\x)-\phi_k(\z^k,\x) \geq \sigma\eta \|x-\z^k\|$.

Since $\z^k\to \x$, using axiom $(M_2)$ there exist $\epsilon_k\to 0^+$ such that
$f(\z^k)-\phi(\z^k,\x) \leq \epsilon_k \|x-\z^k\|$. But then
\[
\widetilde{\rho}_k = \rho_k + \frac{f(\z^k)-\phi(\z^k,\x)}{f(\x)-\phi_k(\z^k,\x)} \leq
\rho_k + \frac{\epsilon_k\|x-\z^k\|}{\sigma\eta \|x-\z^k\|}=\rho_k + \epsilon_k/(\sigma\eta).
\]
Since $\epsilon_k\to 0$, $\rho_k<\gamma$, we have $\limsup_{k\to\infty} \widetilde{\rho}_k
\leq \gamma < \widetilde{\gamma}$, contradicting the fact that
$\widetilde{\rho}_k > \widetilde{\gamma}$ for infinitely many $k$. Hence $\| \g_k\|\geq \eta > 0$
was impossible.

Select $k\in \mathcal K$ such that $\g_k\to 0$. Write $\g_k=\p_k+\q_k$
with $\p_k\in \partial_1 \phi_k(\y^k,\x)$ and $q_k\in N_C(\y^k)$. Using the boundedness of the $\y^k$
extract another subsequence $k\in \mathcal K'$ such that $p_k\to p$, $q_k\to q$. Since $\y^k\to \x$,
we have $q\in N_C(\x)$. We argue that $p\in \partial f(\x)$. Indeed,
for any test vector $\h$ the subgradient inequality gives
\[
p_k^\top h\leq \phi_k(\y^k+\h,\x)-\phi_k(\y^k,\x) \leq
\phi(\y^k+\h,\x)-\phi_k(\y^k,\x).
\]
Since $\phi_k(\y^k,\x)\to f(\x)=\phi(\x,\x)$, passing to the limit gives
\[
\p^\top \h\leq \phi(\x+\h,\x)-\phi(\x,\x),
\]
proving $\p\in\partial_1 \phi(\x,\x)\subset \partial f(\x).$ This proves
that $\x$ is a critical point of (\ref{program}).
\end{proof}

\subsection{Convergence of the outer loop}
In this section we prove our main convergence result. 

\begin{theorem}
\label{theorem1}
Suppose $f$ has a strict first-order model $\phi$.  Let $\x^1\in C$ be such that
$\{\x\in C: f(\x) \leq f(\x^1)\}$ is bounded. Let $\x^j\in C$ be the sequence of iterates
generated by Algorithm {\rm \ref{algo1}}. Then every accumulation point
$\x^*$ of the $\x^j$ is a critical point of {\rm (\ref{program})}.
\end{theorem}

\begin{proof}
1)
Without loss we consider the case where the algorithm generates an infinite sequence $\x^j\in C$
of serious iterates. Suppose that at  outer loop counter $j$ the inner loop finds a successful trial
step at inner loop counter $k_j$, that is, $\z^{k_j} = \x^{j+1}$, where the corresponding
solution of the tangent program is $\tilde{\x}^{j+1}=\y^{k_j}$. Then $\rho_{k_j}\geq \gamma$, which means
\begin{equation}
\label{descent1}
f(\x^j) - f(\x^{j+1}) \geq \gamma \left( f(\x^j) - \phi_{k_j}(\x^{j+1},\x^j) \right).
\end{equation}
Moreover, by condition (\ref{trial}) we have $\|\tilde{x}^{j+1}-x^j\|\leq M\|x^{j+1}-x^j\|$ and
\begin{equation}
\label{descent1a}
f(\x^j)-\phi_{k_j}(\x^{j+1},\x^j) \geq \theta \left( f(\x^j)-\phi_{k_j}(\tilde{\x}^{j+1},\x^j)  \right),
\end{equation}
and combining (\ref{descent1}) and (\ref{descent1a}) gives
\begin{eqnarray}
\label{descent}
f(\x^j)-f(\x^{j+1}) \geq \gamma\theta \left( f(\x^j)-\phi_{k_j}(\tilde{\x}^{j+1},\x^j)  \right).
\end{eqnarray}
Since $\y^{k_j}=\tilde{\x}^{j+1}$ is a solution of the $k_j$th tangent program (\ref{tangent}) of the
$j$th inner loop, there exist
$\g_j\in \partial \left( \phi_{k_j}(\cdot,\x^j)+i_C \right)(\tilde{\x}^{j+1})$ and a unit normal vector
$\vv_j$ to the ball $B(\x^j,R_{k_j})$ at $\tilde{\x}^{j+1}$ such that 
$$
\g_j + \|\g_j\| \vv_j = 0.
$$
We shall now analyze two types of infinite subsequences, those where the trust-region constraint is active at  $\tilde{\x}^{j+1}$, 
and those where it is inactive. 

2)
Let us start with the simpler  case of an infinite subsequence $\x^j$, $j\in J$, where $\|\x^j-\tilde{\x}^{j+1}\| < R_{k_j}$,
i.e., where the trust-region constraint is inactive.
There exist $\p_j\in \partial_1 \phi_{k_j}(\tilde{\x}^{j+1},\x^j)$ and $\q_j\in N_C(\tilde{\x}^{j+1})$
such that
\[
0 = \p_j +  \q_j.
\]
By the subgradient inequality,
applied to $\p_j\in \partial  \phi_{k_j}(\cdot,\x^j)(\tilde{\x}^{j+1})$, we have
\begin{align*}
-\q_j^\top (\x^j-\tilde{\x}^{j+1})=\p_j^{\top} (\x^j-\tilde{\x}^{j+1})
&\leq
\phi_{k_j}(\x^j,\x^j) - \phi_{k_j}(\tilde{\x}^{j+1},\x^j) \\
&= f(\x^j) - \phi_{k_j}(\tilde{\x}^{j+1},\x^j)
\leq \gamma^{-1} \theta^{-1}\left( f(\x^j)-f(\x^{j+1})\right),
\end{align*}
using (\ref{trial}).
Since $\p_j^{\top} (\x^j-\tilde{\x}^{j+1})=\q_j^\top (\tilde{\x}^{j+1}-\x^j) \geq 0$ by Kolmogoroff's inequality,
we deduce summability $\sum_{j\in J} \p_j^{\top} (\x^j-\tilde{\x}^{j+1})<\infty$, hence 
$\p_j^{\top} (\x^j-\tilde{\x}^{j+1})\to 0$, $j\in J$, and then also $\q_j^\top(\x^j-\tilde{\x}^{j+1})\to 0$. 

Let ${\h}$ be any test vector, then
\begin{align*}
\p_j^{\top} {\h} &\leq \phi_{k_j}(\tilde{\x}^{j+1}+{\h},x^j) - \phi_{k_j}(\tilde{\x}^{j+1},\x^j) \\
&\leq \phi(\tilde{\x}^{j+1}+{\h},\x^j) - f(\x^j) + f(\x^j) - \phi_{k_j}(\tilde{\x}^{j+1},\x^j) \\
&\leq \phi(\tilde{\x}^{j+1}+{\h},\x^j) - f(\x^j) + \gamma^{-1}\theta^{-1} \left( f(\x^j)-f(\x^{j+1}) \right).
\end{align*}
Now let ${\h}'$ be another test vector and put ${\h} = \x^j-\tilde{\x}^{j+1}+{\h}'$. Then
on substituting this expression we obtain
\[
\p_j^{\top} (\x^j-\tilde{\x}^{j+1}) + \p_j^{\top} {\h}' \leq \phi(\x^j+{\h}',x^j) - f(\x^j) + \gamma^{-1}\theta^{-1} \left(  f(\x^j)-f(\x^{j+1})\right).
\]
Passing to the limit, we have $\p_j^{\top} (\x^j-\tilde{\x}^{j+1})\to 0$ by the above, and
$f(\x^j)-f(\x^{j+1})\to 0$ by the construction of the descent method. Moreover, $\limsup_{j\in J} \phi(\x^j+{\h}',\x^j) \leq \phi(\x^*+{\h}',\x^*)$
by axiom $(M_3)$ and $\p_j\to \p$ for some $\p$. That shows
\[
\p^{\top} {\h}' \leq \phi(\x^*+{\h}',\x^*) - f(\x^*)=\phi(\x^*+{\h}',\x^*)-\phi(\x^*,\x^*).
\]
Since ${\h}'$ was arbitrary and $\phi(\cdot,\x^*)$ is convex, we deduce
$\p\in \partial_1 \phi(\x^*,\x^*)$, hence $\p\in \partial f(\x^*)$ by axiom $(M_1)$.

Now observe that $\tilde{\x}^{j+1}\to \tilde{\x}$ and $\q_j \to \q\in N_C(\tilde{\x})$. We wish to show that $\q\in N_C(\x^*)$.
Since $\q_j^\top (\x^j-\tilde{\x}^{j+1})\to 0$, we have $\q^\top (\x^*-\tilde{\x})=0$, but $\q\not=0$
and $\x^*-\tilde{\x}\not=0$. Now for any element
$\x\in C$ we have $\q^\top(\tilde{\x}-\x) \geq 0$ by Kolmogoroff's inequality. Hence
$\q^\top (\x^*-\x) = \q^\top(\tilde{\x}-\x) + \q^\top (\x^*-\tilde{\x}) =  \q^\top(\tilde{\x}-\x) \geq 0$,
so Kolmogoroff's inequality holds also at $\x^*$, proving $\q\in N_C(\x^*)$. We have shown that
$0 =\p + \q \in \partial \left( \phi(\cdot,\x^*)+i_C \right)(\x^*)$, hence
$\x^*$ is a critical point of (\ref{program}).

3) Let us now consider the more complicated case of an infinite subsequence,
where $\|\x^j-\tilde{\x}^{j+1}\|=R_{k_j}$ with
$\g_j\not=0$. In other words, the trust-region constraint is active at $\tilde{x}^{j+1}$.
Passing to a subsequence, we may assume $\x^j\to \x^*$,
and we have to show that $\x^*$ is critical.

Let $\uu_j$ be the unit vector $\uu_j = (\tilde{\x}^{j+1}-\x^j )/\|\tilde{\x}^{j+1}-\x^j\|$. Then if the norm $\|\cdot\|$ coincides with the Euclidian norm $|\cdot|$,
we have $\uu_j = \vv_j$. For other norms this is no longer the case,
but for any such norm there exists $\sigma > 0$ such that $\uu_j^\top \vv_j \geq \sigma > 0$
for all $j$. Then
\[
\g_j^\top (\x^j-\tilde{\x}^{j+1}) = - \|\x^j-\tilde{\x}^{j+1}\| \g_j^\top \uu_j
= \|\x^j-\tilde{\x}^{j+1}\| \|\g_j\| \vv_j^\top \uu_j \geq \sigma \| \g_j\| \|\x^j-\tilde{\x}^{j+1}\|.
\]
By the subgradient inequality, and using $\x^j,\tilde{\x}^{j+1}\in C$, we have
\[
\g_j^\top (\x^j-\tilde{\x}^{j+1}) \leq \phi_{k_j}(\x^j,\x^j) - \phi_{k_j}(\tilde{\x}^{j+1},\x^j) = f(\x^j) - \phi_{k_j}(\tilde{\x}^{j+1},\x^j). 
\]
Altogether
\begin{eqnarray}
\label{below}
f(\x^j) - \phi_{k_j}(\tilde{\x}^{j+1},\x^j) \geq \sigma \|\g_j\|\|\x^j-\tilde{\x}^{j+1}\|.
\end{eqnarray}
Combining this with (\ref{descent}) gives
\[
\|\g_j\|\|\x^j-\tilde{\x}^{j+1}\| \leq \sigma^{-1}\gamma^{-1}\theta^{-1}\left( f(\x^j)-f(\x^{j+1}) \right).
\]
Summing both sides from $j=1$ to $j=J$ gives
\[
\sum_{j=1}^J \|\g_j\| \|\x^j-\tilde{\x}^{j+1}\| \leq \sigma^{-1}\gamma^{-1}\theta^{-1} \left( f(\x^1) - f(\x^{J+1})\right).
\]
Since the values $f(\x^j)$ are decreasing
and $\{\x\in C: f(\x)\leq f(\x^1)\}$ is bounded, the sequence $\x^j$ must be bounded.
We deduce that the right hand side is bounded, hence the series on the left converges:
\begin{eqnarray}
\label{sum}
\sum_{j=1}^\infty \|\g_j\| \|\x^j-\tilde{\x}^{j+1}\| < \infty.
\end{eqnarray}
In particular, this implies $\|\g_j\|\|\x^j-\tilde{\x}^{j+1}\| \to 0$.
Using $\|x^j-\x^{j+1}\| \leq M\|\x^j-\tilde{\x}^{j+1}\|$, we also have
$\|\g_j\|\|\x^j-{\x}^{j+1}\| \to 0$.

We shall now have to distinguish two subcases. Either there exists a subsequence $J'\subset J$
such that $R_{k_j} \to 0$ as $j\in J'$, or $R_{k_j} \geq R_0 > 0$ for all $j\in J$.
The second subcase is discussed in 4) below,
the first is handled in 5) - 6).

4)
Let us 
consider the sub-case of an infinite subsequence $j\in J$ where $\|\x^j-\tilde{\x}^{j+1}\|=R_{k_j} \geq R_0 > 0$ for every $j\in J$. 
Going back to (\ref{sum}), we see that we now must have
$\g_j\to 0$, as $\x^j-\tilde{\x}^{j+1}\not\to 0$. Let
us write $\g_j = \p_j + \q_j$, where
$\p_j\in \partial_1 \phi_{k_j}(\tilde{\x}^{j+1},\x^j)$ and $\q_j\in N_C(\tilde{\x}^{j+1})$.
Then 
\[
\p_j^\top (\x^j-\tilde{\x}^{j+1}) \leq \phi_{k_j}(\x^j,\x^j)-\phi_{k_j}(\tilde{\x}^{j+1},\x^j)
\leq \gamma^{-1} \theta^{-1}\left( f(\x^j)-f(\x^{j+1}\right).
\]
Now $\g_j^\top (\x^j-\tilde{\x}^{j+1}) = \p_j^\top(\x^j -\tilde{\x}^{j+1}) + \q_j^\top (\x^j-\tilde{\x}^{j+1}) \leq p_j^\top (\x^j-\tilde{\x}^{j+1})$,
because Kolmogoroff's inequality for $\tilde{\x}^{j+1}\in C$ and $\q_j\in N_C(\tilde{\x}^{j+1})$ gives  $\q_j^\top (\tilde{\x}^{j+1}-\x^j)\geq 0.$
Hence we have
\[
\g_j^\top (\x^j-\tilde{\x}^{j+1}) \leq \p_j^\top (\x^j-\tilde{\x}^{j+1}) \leq \gamma^{-1}\theta^{-1} \left( f(\x^j)-f(\x^{j+1} )\right),
\] 
so $\p_j^\top (\x^j-\tilde{\x}^{j+1})\to 0$, because the lefthand term and the righthand term both converge to 0. 
As a consequence, we also have $\q_j^\top (\x^j-\tilde{\x}^{j+1})\to 0$. 

Now observe that the sequence $\x^j\in C$ is also bounded, because $\{\x\in C: f(\x)\leq f(\x^1)\}$ is bounded
and the $\x^j$ form a descent sequence for $f$. Let us say $\|\x^1-\x^j\|\leq K$ for all
$j$.  We argue that the $\p_j$ are then also bounded.
This can be shown as follows. Let ${\h}$ be a test vector with
$\|{\h}\|=1$. Then
\begin{align*}
\p_j^{\top} {\h} &\leq \phi_{k_j}(\tilde{\x}^{j+1}+{\h},\x^j) - \phi_{k_j}(\tilde{\x}^{j+1},\x^j)\\
&\leq \phi(\tilde{\x}^{j+1}+{\h},\x^j) - m_{0j}(\tilde{\x}^{j+1},\x^j)\\
&= \phi(\tilde{\x}^{j+1}+{\h},\x^j) 
-f(\x^j) - \g_{0j}^\top (\tilde{\x}^{j+1}-\x^j)\\
&\leq C +| f(\x^1) |+ \|\g_{0j}\| \| \x^j-\tilde{\x}^{j+1}\|,
\end{align*}
where $C:= \max\{\phi(u,v): \|u-x^1\|\leq MK+1, \|v-x^1\|\leq K\} < \infty$
and where $\g_{0j}\in \partial f(\x^j)$ by the definition of the exactness plane at $\x^j$.
But observe that $\partial f$ is locally bounded by \cite{RW98}, so
$\|\g_{0j}\|\leq K' < \infty$.
We deduce $\|\p_j\| \leq C + |f(x^1)| + K'(2K+M) < \infty$.
Hence the sequence $\p_j$ is bounded, and since $\g_j=\p_j+\q_j\to 0$ by the above, 
the sequence $\q_j$ is also bounded.

Therefore,  on passing to a subsequence $j\in J'$, we may assume
$\x^j\to \x^*$, $\tilde{\x}^{j+1}\to \tilde{\x}$, $\p_j\to \p$, $\q_j\to \q$. Then $\q\in N_C(\tilde{\x})$.    
Now from the subgradient inequality
\begin{align*}
\p_j^{\top} {\h} &\leq \phi_{k_j}(\tilde{\x}^{j+1}+{\h},\x^j) - \phi_{k_j}(\tilde{\x}^{j+1},\x^j)\\
&\leq \phi(\tilde{\x}^{j+1}+{\h},\x^j) - f(\x^j) + f(\x^j)-\phi_{k_j}(\tilde{\x}^{j+1},\x^j) \\
&\leq \phi(\tilde{\x}^{j+1}+{\h},\x^j) - \phi(\x^j,\x^j) + \gamma^{-1} \theta^{-1}\left( f(\x^j)-f(\x^{j+1}) \right),
\end{align*}
where we use (\ref{trial}),  $\phi_{k_j}\leq \phi$,  and acceptance $\rho_{k_j}\geq \gamma$, and where
the test vector ${\h}$ is arbitrary. Let ${\h}'$ another test vector and put ${\h}=\x^{j}-\tilde{\x}^{j+1}+{\h}'$.
Substituting this gives
\begin{equation}
\label{last}
\p_j^{\top} (\x^j-\tilde{\x}^{j+1}) + p_j^{\top} {\h}' \leq
\phi(\x^j + {\h}',\x^j) - \phi(\x^j,\x^j) + \gamma^{-1}\theta^{-1}\left( f(\x^j)-f(\x^{j+1})  \right).
\end{equation}
Now $\p_j^{\top} (\x^j-\tilde{\x}^{j+1})=
(\p_j+\q_j)^\top (\x^j-\tilde{\x}^{j+1}) + \q_j^\top (\tilde{\x}^{j+1}-\x^j) \geq (\p_j+\q_j)^\top (\x^j-\tilde{\x}^{j+1})$ using Kolmogoroff's condition
for $\q_j\in N_C(\tilde{\x}^{j+1})$. Therefore, on passing to the limit in (\ref{last}),
using $(\p_j+\q_j)^\top (\x^j-\tilde{\x}^{j+1})\to 0$, $f(\x^j)-f(\x^{j+1})\to 0$, $\p_j\to \p$ and
$\limsup_{j\in J'} \phi(\x^j+{\h}',\x^j) \leq \phi(\x^*+{\h}',\x^*)$, which follows from axiom $(M_3)$, we find
\[
\p^{\top} {\h}' \leq \phi(\x^*+{\h}',\x^*)-\phi(\x^*,\x^*).
\]
Since ${\h}'$ was arbitrary, we deduce $\p\in \partial_1 \phi(\x^*,\x^*)$, and by axiom
$(M_1)$, $\p\in \partial f(\x^*)$. 

It remains to show $\q\in N_C(\x^*)$. Now recall that $\q_j^\top (\x^j-\tilde{\x}^{j+1})\to 0$ was
shown at the beginning of part 4), so $\q^\top (\x^*-\tilde{\x})=0$. Given any
test element $\x\in C$, Kolmogoroff's inequality for $\q\in N_C(\tilde{\x})$  gives $\q^\top(\tilde{\x}-\x) \geq 0$. But then 
$\q^\top(\x^*-\x) = \q^\top (\tilde{\x}-\x) + \q^\top (\x^*-\tilde{\x}) = \q^\top (\tilde{\x}-\x)\geq 0$, so Kolmogoroff's inequality also
holds for $\q$ at $\x^*$, proving $\q\in N_C(\x^*)$.  

With $\q\in N_C(\x^*)$ and $\g=\p+\q=0$, we have shown that $\x^*$ is a critical point of (\ref{program}).
That settles the case where the trust-region radius is active and bounded away from 0. 

5)
It remains to discuss the most complicated sub-case of an infinite subsequence $j\in J$, where the
trust-region constraint is active and $R_{k_j}\to 0$. This needs two sub-sub-cases.
The first of these is a sequence $j\in J$ where in each $j$th outer loop the trust-region radius 
was reduced at least once. The second sub-sub-case are infinite subsequences 
where  the trust-region radius stayed frozen ($R_j^\sharp = R_{k_j}$) throughout the $j$th inner loop
for every $j\in J$. This is discussed in 6) below.

Let us first consider the case of an infinite sequence $j\in J$  where $R_{k_j}$ is active
at $\tilde{x}^{j+1}$,  and  $R_{k_j}\to 0$, $j\in J$, such that
during the $j$th inner loop the trust-region radius was reduced at least once.
Suppose this happened  the last time before acceptance at inner loop counter $k_j-\nu_j$.
Then for $j\in J$,
\[
R_{k_j} = R_{k_j-1} = \dots = R_{k_j-\nu_j} =\textstyle \frac{1}{2}R_{k_j-\nu_j-1}.
\]
By step 7 of the algorithm, that implies
\[
\widetilde{\rho}_{k_j-\nu_j} \geq \widetilde{\gamma}, \quad \rho_{k_j-\nu_j} < \gamma.
\]
Now $\|\x^{j+1}-\x^j\| \leq R_{k_j}$ and $\|\z^{k_j-\nu_j}-\x^j\|\leq R_{k_j-\nu_j-1} = 2R_{k_j}$, hence
$\x^{j+1}-\z^{k_j-\nu_j} \to 0$, $\x^j-\z^{k_j-\nu_j}\to 0$, $j\in J''$. From axiom $(\widetilde{M}_2)$
we deduce that there exists a sequence $\epsilon_j\to 0^+$ such that
$$
f(\z^{k_j-\nu_j}) \leq \phi(\z^{k_j-\nu_j},\x^j) + \epsilon_j \|\z^{k_j-\nu_j}-\x^j\|.
$$
By the definition of the aggregate subgradient 
$\widetilde{\g}_j \in \partial\left(\phi_{k_j-\nu_j}(\cdot,\x^j)+i_C  \right)(\y^{k_j-\nu_j})$ and Lemma \ref{lem_cauchy} we have 
$f(\x^j)-\phi_{k_j-\nu_j}(\z^{k_j-\nu_j},\x^j) \geq \sigma \|\widetilde{\g}_j\| \| \x^j-\z^{k_j-\nu_j}\|$. 

Recall that $\x^j\to \x^*$  and that we
have to show that $\x^*$ is critical. It suffices to show that there is a subsequence $j\in J'$ with $g_j\to 0$.
Assume on the contrary  that $\|\widetilde{\g}_j\|\geq \eta > 0$ for every $j\in J$. Then
\[
f(\x^j) - \phi_{k_j-\nu_j}(\z^{k_j-\nu_j},\x^j) \geq \eta\sigma \|\z^{k_j-\nu_j}-\x^j\|.
\]
Now
\[
\widetilde{\rho}_{k_j-\nu_j} = \rho_{k_j-\nu_j}+ \frac{f(\z^{k_j-\nu_j}) - \phi(\z^{k_j-\nu_j},\x^j)}{f(\x^j)-\phi_{k_j-\nu_j}(\z^{k_j-\nu_j},\x^j)}
\leq
\rho_{k_j-\nu_j}+ \frac{\epsilon_j \|\z^{k_j-\nu_j}-\x^j\|}{\eta \|\z^{k_j-\nu_j}-\x^j\|} < \widetilde{\gamma}
\]
for $j\in J$ sufficiently large, contradicting $\widetilde{\rho}_{k_j-\nu_j} \geq \widetilde{\gamma}$. This shows that
there must exist a subsequence $J'$ such that $\widetilde{\g}_j\to 0$, $j\in J'$. Passing to the limit
$j\in J'$, this shows
$0\in \partial \left( \phi(\cdot,\x^*)+i_C \right)(\x^*)$, hence $\x^*$ is critical for (\ref{program}).

6)
Now consider an infinite subsequence $j\in J$ where $\x^j\to \x^*$,  the trust-region radius $R_{k_j}$ was active
at $\tilde{x}^{j+1}$ when $\x^{j+1}$ was accepted, $R_{k_j}\to 0$,
but during the $j$th inner loop the trust-region radius was never reduced. In the classical case this can only
happen when $\x^{j+1}$ at $j$ is immediately accepted, but with bundling
this could also happen when the inner loop adds cutting planes for a time, while the test in step 7 keeps
$R_{k+1}=R_k$ in the inner loop. Since $R_{k_j}\to 0$, the work to bring the radius to
0 must be  put about somewhere else. For every $j\in J$ define $j' \in \mathbb N$ to be the largest index $j' < j$
such that in the $j'$th inner loop, the trust-region radius was reduced at least once.
Let $J' = \{j': j\in J\}$, where we understand $j\mapsto j'$ as a function. Passing to a subsequence of $J,J'$,
we may assume that $\x^{j'}\to \x'$ and $g_{j'} \to 0$, because the sequence $J'$ corresponds to one of
the cases discussed in parts 2) - 5). Passing to jet another subsequence, we may arrange that the sequences
$J,J'$ are interlaced. That is, $j' < j < j'^{+} < j^+ < j'^{++} < j^{++} < \dots \to \infty$. This is because
$j'$ tends to $\infty$ as a function of $j$.

Now assume that there exists $\eta > 0$ such that $\|\g_j\|\geq \eta$ for all $j\in J$. Then since $\x^j \to \x^*$,
we also have $\x^{j+1}\to \x^*$.
Fix $\epsilon > 0$ with $\epsilon < \eta$. For $j\in J$ large enough we have $\|\g_{j'}\| < \epsilon$, because
$\g_{j'}\to 0$, $j'\in J'$, and as $j$ gets larger, so does $j'$. That means in the interval $[j',j)$ there exists an index
$j'' \in \mathbb N$ such that
\[
\|\g_{j''}\| < \epsilon, \quad \| \g_{i}\| \geq \epsilon \mbox{ for all } i = j'' + 1,\dots,j.
\]
The index $j''$ may coincide with $j'$, it might also be larger, but it precedes $j$. In any case, $j \mapsto j''$ is again a function on $J$ and defines another infinite index set
$J''$ still interlaced with $J$.

Now recall from part 3), estimate (\ref{sum}), and  $\| \x^j-\x^{j+1}\|\leq M\|x^{j}-\tilde{\x}^{j+1}\|$, that for some constant $c>0$
\[
\sum_{i=j''+1}^{j} \|\g_i\| \|\x^i - \x^{i+1}\| \leq c \left(  f(\x^{j''+1})-f(\x^{j+1}) \right) \to 0 \qquad (j\in J, j\to\infty, j\mapsto j'').
\]
Since by construction $\|g_i\| \geq \epsilon$ for all $i \in [j''+1,\dots,j]$, and that for all $j\in J$, the sequence
$\sum_{i=j''+1}^j \|x^i-\x^{i+1}\|\to 0$ converges as $j\in J, j\to \infty$, and by the triangle inequality,
$\x^{j''+1}-\x^{j+1}\to 0$. Therefore $\x^{j''+1}\to \x^*$.  Since  $g_{j''}\in \partial (f+i_C)(\x^{j''+1})$, 
passing to yet another subsequence and using upper semi-continuity of the subdifferential,
we get $g_{j''} \to \tilde{g} \in \partial (f+i_C)(\x^*)$. Since $\|g_{j''}\| < \epsilon$, we have $\|\tilde{g}\| \leq \epsilon$.
It follows that $\partial (f+i_C)(\x^*)$ contains an element of norm $\leq \epsilon$.  
As $\epsilon < \eta$ was arbitrary, we conclude that $0\in \partial (f+i_C)(\x^*)$.
That settles the remaining case.
\end{proof}

\section{Stopping test}
A closer look at the convergence proof indicates stopping criteria for algorithm \ref{algo1}.
As is standard in bundle methods, step 2 is not executed as such but delegated to the inner loop.
When a serious step $\x^{j+1}$ is accepted, we apply the tests
\[
\frac{\|\x^j-\x^{j+1}\|}{1+\|\x^j\|} < {\rm tol}_1, \quad \frac{f(\x^j)-f(\x^{j+1})}{1+| f(\x^j)|} < {\rm tol}_2
\]
in tandem with 
\[
\frac{\min \{\|\g_j\|, \|g_{j'}\|  \}}{1+ | f(\x^j)|} < {\rm tol}_3.
\]
Here $\g_j$ is the aggregate subgradient at acceptance. In the case treated in part 6) of the proof 
we had to consider the largest index $j' < j$, where the trust-region radius was reduced for the last time.
If in the inner loop at $\x^j$ leading to $\x^{j+1}$ the trust-region radius was not reduced, we have to consider both
aggregates, otherwise $\|\g_j\|/(1+\| \x^j\|) < {\rm tol}_3$ suffices.
If the three criteria are satisfied, then we return $\x^{j+1}$ as optimal.

On the other hand, when the inner loop has difficulties finding a new serious iterate, and if a maximum number
$k_{\max}$ is exceeded, or if for $\nu_{\max}$ consecutive steps
\[
\frac{\|\x^j-\z^k\|}{1+\|\x^j\|} < {\rm tol}_1, \quad \frac{f(\x^j)-f(\z^k)}{1+| f(\x^j)|} < {\rm tol}_2
\]
in tandem with 
\[
\frac{ \|\g_k\|  }{1+ | f(\x^j)|} < {\rm tol}_3
\]
are satisfied, where $\g_k$ is the aggregate subgradient at $\y^k$, then the inner loop is stopped
and $\x^j$ is returned as optimal. In our tests we use $k_{\max}=50$, $\nu_{\max} = 5$,
${\rm tol}_1 = {\rm tol}_2 = 10^{-5}$, ${\rm tol}_3 = 10^{-6}$. Typical values in algorithm \ref{algo1} are
$\gamma =0.0001$, $\widetilde{\gamma}= 0.0002$, $\Gamma = 0.1$.

\section{Applications}
\label{sect_applications}
In this section we highlight the potential of the model-based trust-region
approach by presenting several applications.

\subsection{Full model versus working model}
\label{full}
Our convergence theory covers the specific case $\phi_k = \phi$, which we call the
{\em full model case}.
Here the algorithm simplifies, because
cutting planes are redundant, so that step 6 becomes obsolete. Moreover,
in step 7 the quotient $\widetilde{\rho}_k$ always equals $1$, so the only action taken is reduction 
of the trust-region radius.  This is now close to  the rationale of the classical trust-region method.

\subsection{Natural model}
For a composite function
$f=g\circ F$ with $g$ convex and $F$ of class $C^1$
the {\em natural model} is $\phi(\y,\x) = g\left( F(\x) + F'(\x)(\y-\x) \right)$, because it is strict and can 
be used in algorithm \ref{algo1}. 
In the full model case $\phi_k = \phi$, our algorithm reduces to the algorithm
of Ruszczy\'nski \cite[Chap. 7.5]{rud} for composite
nonsmooth functions.

\subsection{Spectral model}
\label{spectral}
An important field of applications, where the natural model often comes into action,
are eigenvalue optimization problems
\begin{eqnarray}
\label{eigen}
\begin{array}{ll}
\mbox{minimize} & \lambda_1 \left( \mathcal F(\x) \right)\\
\mbox{subject to}&\x\in C
\end{array}
\end{eqnarray}
where $\mathcal F:\mathbb R^n\to \mathbb S^m$ is a class $C^1$-mapping into the space of
$m\times m$ symmetric or Hermitian matrices $\mathbb S^m$, and $\lambda_1(\cdot)$  the maximum eigenvalue
function on $\mathbb S^m$, which is convex but nonsmooth. Here the natural model 
is $\phi(\y,\x) = \lambda_1\left( \mathcal F(\x) + \mathcal F'(\x)(\y-\x) \right)$, where $\mathcal F'$ is the differential of $\mathcal F$.
Note that nonlinear semidefinite programs
\begin{eqnarray}
\label{nlsdp}
\begin{array}{ll}
\mbox{minimize} & f(\x) \\
\mbox{subject to}& \mathcal F(\x)\preceq 0\\
&\x\in C
\end{array}
\end{eqnarray}
are special cases of (\ref{eigen}) if we use exact penalization
and write (\ref{nlsdp}) in the form
\begin{eqnarray*}
\begin{array}{ll}
\mbox{minimize} & f(\x) + c \max\left\{0,\lambda_1 \left( \mathcal F(\x) \right)\right\}\\
\mbox{subject to}&\x\in C
\end{array}
\end{eqnarray*}
with a suitable $c>0$. Namely,
this new objective may be written as the maximum eigenvalue of the mapping
$$
\mathcal F^\sharp(\x) = \left[ \begin{array}{cc} f(\x)  & 0 \\ 0 & f(\x)I_m + c\mathcal F(\x) \end{array} \right] \in \mathbb S^{1+m}.
$$

Let us apply the bundling idea to
(\ref{eigen}) using the natural model $\phi$.  Here we may build working models $\phi_k$
generated by infinite sets $\mathcal G_k$ of cuts $(a,g)$ from $\phi$, and still arrive at a computable
tangent program. Indeed, suppose
$\y^k$ is a null step  at serious iterate $\x$. According to
step 6 of algorithm \ref{algo1} we have to generate one or several cutting planes
at $\y^k$. This means we have to
compute $g_k\in \partial \lambda_1\left( \mathcal F(\x) + \mathcal F'(\x)(\cdot-\x)\right)(\y^k)$. 
Now by the generalized chain rule the subdifferential of the composite
function $\y \mapsto \lambda_1\left( \mathcal F(\x)+\mathcal F'(\x)(\y-\x)\right)$
at $\y$ is $\mathcal F'(\x)^* \partial \lambda_1 \left( \mathcal F(\x)+\mathcal F'(\x)(\y-\x) \right)$,
where $\partial \lambda_1$ is now the convex subdifferential 
of $\lambda_1$ in matrix space $\mathbb S^m$, i.e.,
\[
\partial \lambda_1(X) = \{G \in \mathbb S^m: G \succeq 0, {\rm tr}(G)=1, G \bullet X = \lambda_1(X)\}
\]
with $X\bullet Y = {\rm tr}(XY)$ the scalar product in $\mathbb S^m$. Here $\mathcal F'(\x)^*:\mathbb S^m\to \mathbb R^n$ is the
adjoint of the linear operator $\mathcal F'(\x)$. It follows that every subgradient $g$ of the composite function is of the form
\begin{equation}
\label{sub_lam}
g = \mathcal F'(\x)^* G, \quad G \in \partial\lambda_1\left( \mathcal F(\x)+\mathcal F'(\x)(\y-\x) \right).
\end{equation}
The
corresponding $a$ is $a=\lambda_1\left( \mathcal F(\x)+\mathcal F'(\x)(\y-\x)  \right)+g^\top (\x-\y)$.
As soon as the maximum eigenvalue $\lambda_1(X)$ has multiplicity
$>1$, the set $\partial\lambda_1(X)$ is not singleton, and we may therefore
add the entire subdifferential to the new set $\mathcal G_{k+1}$.

Let $\y^k$ be a null step, and let
$Q_r$ be an $m\times t_k$ matrix whose $t_k$ columns form an orthogonal basis
of the maximum eigenspace of $\mathcal F(\x) + \mathcal F'(\x)(\y^k-\x)$.  Let $Y_k$ be a $t_k\times t_k$-matrix
with $Y_k=Y_k^\top$, $Y_k\succeq 0$, tr$(Y_k)=1$, then subgradients 
(\ref{sub_lam}) are of the form $G_k = Q_kY_kQ_k^\top$. Therefore all pairs 
$(a_r,\g_r(Y_r))\in \mathcal G_k$ are of the form
\[
a_r = \lambda_1\left( \mathcal F(\x)+\mathcal F'(\x)(\y^r-\x) \right),
\quad \g_r(Y_r) = \mathcal F'(\x)^* G_r, \quad G_r=Q_rY_rQ_r^\top,
\]
indexed by $Y_r \succeq 0$, tr$(Y_r)=1$, $Y_r\in \mathbb S^{t_r}$ stemming from older null steps
$r=1,\dots,k$. The trust-region tangent program
is then
\begin{eqnarray*}
\begin{array}{ll}
\mbox{minimize} &\displaystyle \max_{r=1,\dots,k}  a_r + \lambda_1
 \left(Q_r \mathcal F'(\x) (\y-\y^r) Q_r^\top   \right) \\
\mbox{subject to}&\y\in C, \; \|\y-\x\| \leq R
\end{array}
\end{eqnarray*}
This is  a linear semidefinite program
if a polyhedral or a conical norm is used, and if $C$ is a convex semidefinite constraint set. 

We can go one step
further and consider semi-infinite maximum eigenvalue problems
as in \cite{ANP}, as this has scope for applications in automatic control. It
allows us for instance to optimize the $H_\infty$-norm, or more general IQC-constrained
programs, see \cite{IQC}. 

\subsection{Standard model}
\label{sect_standard}
The most straightforward choice of a model is
the {\em standard model}
\[
\phi^\sharp(\y,\x) = f(\x) + f^\circ(\x,\y-\x),
\]
as it gives a direct substitute for the first-order Taylor expansion of $f$ at $\x$. 
Here the full model tangent program (\ref{tangent}) has the specific form
\begin{eqnarray}
\label{cauchy}
\begin{array}{ll}
\mbox{minimize} &f(\x)+f^\circ(\x,\y-\x)\\
\mbox{subject to} & \y\in C\\&\|\y-\x\| \leq R_k
\end{array}
\end{eqnarray}
and if a polyhedral working model $\phi_k^\sharp$ is used to approximate $\phi^\sharp$
via bundling, then we get an even simpler tangent program
of the form
\begin{eqnarray}
\label{tangent_simple}
\begin{array}{ll}
\mbox{minimize} &f(\x)+ \displaystyle\max_{i=1,\dots,k} \g_i^\top (\y-\x)\\
\mbox{subject to} & \y\in C\\
&\|y-\x\| \leq R_k
\end{array}
\end{eqnarray}
where $\g_i\in \partial f(\x)$.
If a polyhedral norm is used and $C$ is a polyhedron, then (\ref{tangent_simple})
is just a linear program, which makes this line attractive computationally. 

\begin{remark}
Consider the unconstrained case $C=\mathbb R^n$ with $\phi_k^\sharp=\phi^\sharp$,  then
$\y^k= \x - R_k \g(\x)/\|\g(\x)\|$, where
$\g(\x)= \underset{g\in \partial f(\x)} {\mathrm{argmin}} ~\{ \|\g\|: g\in \partial f(\x)\}$,
and this is the nonsmooth steepest descent step of length $R_k$ at $\x$. In classical trust-region algorithms
the steepest descent step of length $R_k$ is often chosen as the first-order Cauchy step. 
\end{remark}

This raises the following
natural question. Can we use the solution of $y^k$ of (\ref{cauchy}), or  (\ref{tangent_simple}),
as a nonsmooth Cauchy point? 
Since we do not want to keep the reader on the tenterhooks too long, here is the answer:
{\em no we can't}.
Namely, 
in order to be allowed to use the standard model in Algorithm \ref{algo1}, and 
the solution of (\ref{cauchy}), (\ref{tangent_simple}) as a Cauchy point for other models,
$\phi^\sharp$ has to be strict, because this is required in Theorem \ref{theorem1}.
A sufficient condition for strictness of $\phi^\sharp$ is given in \cite{Noll2012}. 
We need the following

\begin{definition}
[Spingarn \cite{Spi81}, Rockafellar-Wets \cite{RW98}]
A locally Lipschitz function $f:\mathbb R^n \to \mathbb R$ is lower-$C^1$ at  $\x_0\in \mathbb R^n$
if there exist a compact space $\mathbb K$, a neighborhood $U$ of $\x_0$, and a mapping $F:\mathbb R^n \times \mathbb K \to \mathbb R$ such that
\begin{equation}
\label{lower}
f(\x) = \max_{\y\in\mathbb  K} F(\x,\y)
\end{equation}
for all $\x \in U$, and $F$ and $\partial F/\partial \x$ are jointly continuous. The function $f$ is said to be upper-$C^1$ at $\x_0$ if
$-f$ is lower-$C^1$ at $\x_0$.
\hfill $\square$
\end{definition}

\begin{lemma}
(See {\rm \cite{Noll2012}}).
Suppose $f$ is locally Lipschitz and upper $C^1$. Then the standard model $\phi^\sharp$ of $f$
is strict. \hfill $\square$
\end{lemma}

\begin{example}
\label{note}
The lightning function $f:\mathbb R \to \mathbb R$ 
in \cite{klatte} is an example where $\phi^\sharp$ is strict, but $f$ is not upper $C^1$. 
It is Lipschitz 
with constant $1$ and has $\partial f(x) = [-1,1]$ for every $x$. The standard model of
$f$ is strict, because for all $x,y$ there exists $\rho = \rho(x,y)\in [-1,1]$ such that
\begin{multline*}
f(y) = f(x) + \rho |y-x|  \leq f(x) + {\rm sign}(y-x)(y-x) \\
\leq f(x) + f^\circ(x,y-x)= \phi^\sharp(x,y-x),
\end{multline*}
using the fact that sign$(y-x)\in \partial f(x)$.  At the same time $f$ is certainly not upper-$C^1$, 
because it is not semi-smooth in the sense of \cite{mifflin}.
\end{example}

When using the standard model $\phi^\sharp$ in Algorithm \ref{algo1}, we 
expect the trust-region method to coincide with its classical
antecedent, or at least, to be very similar to it.
But we expect more! Let $\mathscr S$ be the class of nonsmooth locally Lipschitz functions $f$ which have a strict
standard model $\phi^\sharp$. Suppose a subclass $\mathscr S'$ of $\mathscr S$ leads to simplifications
of algorithm \ref{algo1} which reduce it to its classical alter ego. Then we have a theoretical
justification to say  that functions $f\in \mathscr S'$, even though nonsmooth,  can be optimized {\em as if they were smooth}.

Following Borwein and Moors \cite{borwein-moors},  a function $f$ is called {\em essentially 
smooth} if it is locally Lipschitz and strictly differentiable almost everywhere.
The lightning function of example \ref{note} is a pathological case, which is differentiable almost everywhere, 
but nowhere strictly differentiable.  In practice
we expect nonsmooth functions 
to be essentially smooth. This is for instance the case
for semi-smooth functions in the sense of \cite{mifflin}, for
arc-wise essentially smooth functions, or for pseudo-regular
functions in the sense of \cite{borwein-moors}. 

\begin{proposition}
\label{almost}
Let $f$ be essentially smooth. Let $\x^1\in C$ be such that 
$\{\x\in C: f(\x)\leq f(\x^1)\}$ is bounded.
Suppose the
standard model $\phi^\sharp$ is used  in algorithm {\rm \ref{algo1}}. 
Let trial points $\z^k\in C$ satisfying {\rm (\ref{trial})}  in step {\rm 4}  are drawn at random and independently according to a
continuous probability distribution on $C$.
Then with probability one the steps of the
algorithm are identical with the steps of the classical trust-region algorithm. Moreover, if $\phi^\sharp$
is strict, then every accumulation point of the sequence $\x^j$ is critical.
\end{proposition}

\begin{proof}
Since there exists a full neighborhood $U$ of $\y^k$ such that every $\z^k\in U \cap C$ is a valid trial point,
and since the elements in $U\cap C$ are with probability 1 points of strict differentiability, the entire sequence
$\x^j$ consists with probability 1 of points of strict differentiability.
\end{proof}

Note that we should not expect the $\y^k$ themselves to be points of  differentiability, let alone strict differentiability.
In fact the $\y^k$ will typically lie in a set of measure 0. For instance, if $C$ is a polyhedron,
then $\y^k$ is typically a vertex of $C$, or a vertex of the polyhedron of the linear program (\ref{tangent_simple}).

Proposition \ref{almost} applies
in particular when $f$ is upper $C^1$, because upper $C^1$-functions are essentially smooth.
However,
for upper $C^1$ functions we have the following
stronger result. A similar observation in the context of bundle methods was first made in \cite{dao}.

\begin{lemma}
\label{upper}
Suppose $f$ is locally Lipschitz and upper-$C^1$ and the standard model $\phi^\sharp$ is used in algorithm {\rm \ref{algo1}}.
Then we can choose the cutting plane $m_k(\cdot,\x)=f(\x)+\g_k^\top(\cdot-\x)$
in {\rm step 6} with $\g_k\in \partial f(\x)$ arbitrarily, because $f^\circ(\x,\z^k-\x) - g_k^\top(\z^k-\x) \leq \epsilon_k\|\z^k-\x\|$
holds automatically for some $\epsilon_k\to 0^+$ in the inner loop at $\x$, and
$f^\circ(\x^{j},\x^{j+1}-\x^j) - \g_j^\top(\x^{j+1}-\x^j)\leq \epsilon_j\|\x^{j+1}-\x^j\|$ holds automatically for some $\epsilon_j\to 0^+$
in the outer loop. 
\end{lemma}

\begin{proof}
Daniilidis and Georgiev \cite[Thm. 2]{dani} prove that an upper $C^1$
function is super-monotone at $\x$  in the following sense: For every $\epsilon > 0$
there exists $\delta>0$ such that
$(\g_1-\g_2)^\top (\x_1-\x_2) \leq \epsilon \|\x_1-\x_2\|$ for all $\x_i\in U$ and
$\g_i\in \partial f(\x_i)$. Hence for sequences $\x^j,\y^j \to \x$ we find $\epsilon_j\to 0^+$
such that $(\g_j^*-\g_j)^\top (\x^j-\y^k)\leq \epsilon_j\|\y^j-\x^j\|$ for all $\g_j^*\in \partial f(\y^j)$, $\g_j\in \partial f(\x^j)$.
Choosing $\g_j^*$ such that
$f^\circ(\x^j,\y^j-\x^j)=\g_j^{*\top}(\y^j-\x^j)$ then gives the result.
\end{proof}

As a consequence we have the following

\begin{theorem}
Suppose $f$ is upper-$C^1$, $\x^1\in C$, and $\{\x\in C: f(\x) \leq f(\x^1)\}$ is bounded. Suppose
the classical trust-region algorithm is used, that is,
the only cutting plane in step {\rm 6} chosen at $\x$ is an arbitrarily  exactness plane, and  in step {\rm 7}
the trust-region radius is reduced whenever a null step occurs. Then every accumulation point
of the sequence of serious iterates $\x^j$ is a critical point of {\rm (\ref{program})}. Moreover, if  $f$ satisfies
the Kurdyka-\L ojasiewicz inequality, then the  $\x^j$ converge to a single critical point 
$\x^*$ of $f$.
\end{theorem}

\begin{proof}
By Lemma \ref{upper} the  proof of Theorem \ref{theorem1}
applies regardless how we choose cutting planes from $\phi^\sharp$. We exploit this by choosing them
in the simplest possible way, namely we take only
one exactness plane and keep it all the time. 
If $f$ is differentiable at $\x$ then our only choice is $m(\cdot,\x)=f(\x)+\nabla f(\x)^\top (\cdot-\x)$, otherwise we
take $m(\cdot,\x)=f(\x)+\g^\top(\cdot-\x)$ with an arbitrary $\g\in \partial f(\x)$.
This makes step 6 redundant and reduces step 7 to the usual
modification of the trust-region radius.  And this is now just the classical trust-region strategy,
for which we then  have subsequence convergence
by Theorem \ref{theorem1}.

It remains to show that under the Kurdyka-\L ojasiewicz inequality the $\x^j$ converge even  to a single limit.  
This can be based on the technique of \cite{absil,attouch,Noll2012}.
\end{proof}

\begin{remark}
An axiomatic approach to trust-region methods is Dennis {\em et al.} \cite{dennis}, 
and the idea is adopted
in \cite[Chap. 11]{conn}. The difference with our approach is that $\phi$ in \cite{dennis,conn}
has to be jointly
continuous, while we use the weaker axiom $(M_3)$, and that their $f$ 
has to be regular, which precludes the use of the standard model $\phi^\sharp$, hence
makes it impossible to use the Cauchy point.
Bundling is not discussed in these approaches.

On the other hand,
the authors of
\cite{dennis}, \cite{conn} do
allow non-convex models, while in our approach $\phi(\cdot,\x)$ is convex
because we want to assure a computable tangent program,
and be able to draw cutting planes. Convexity of $\phi(\cdot,\x)$ could be relaxed to
$\phi(\cdot,\x)$ being
lower $C^1$. For that the downshift idea \cite{mifflin,Noll2010} would have to be used. 
\end{remark}

\subsection{Delamination problem}
Contact mechanics is a domain where nonsmooth optimization programs arise 
frequently. When potential energy is minimized under non-monotone friction laws,
then programs with lower-$C^1$ functions arise.  On the other hand,
quasi-static delamination problems lead to minimization of upper-$C^1$
criteria, see \cite{gwinner,raous,adly} for more information. 

\subsection{Model for splitting}
Suppose we wish to optimize a function
$f = g + h$ where $g$ is differentiable and $h$ is convex. Then
a model $\phi$ for $f$ is $\phi(\y,\x) = g(\x) + \nabla g(\x)^\top (\y-\x) + h(\y)=\phi^\sharp_g(\y,\x) + h(\y)$. 
Indeed, for the differentiable $g$ the first-order Taylor expansion is natural, and  the
convex $h$ is its own strict model. Cutting planes are now 
sums of cutting planes of the two model components.
Algorithm \ref{algo1} based on $\phi$ could then be an alternative to
a splitting technique, in particular, as ours carries over easily to the case 
when $h$ is lower-$C^2$.

\subsection{Failure of the Cauchy point}
\label{counter}
We will show by way of an example that
the classical trust-region approach based on the Cauchy point fails in the nonsmooth case.
We operate algorithm \ref{algo1} with the full standard model $\phi^\sharp$, compute the Cauchy point $\y^k$
via (\ref{cauchy}) based on the Euclidian norm, and use $\z^k=\y^k$ as the trial step.
This corresponds essentially to a classical first-order trust-region method.

The following example adapted from \cite{HuL:93}
can be used to show the difficulties with this classical scheme.
We define a convex piecewise affine function $f:\mathbb R^2 \to \mathbb R$ as
\[
f(\x) = \max\{f_0(\x),f_{\pm 1}(\x),f_{\pm 2}(\x) \}
\]
where $\x = (x_1,x_2)$ and
\[
f_0(\x)= -100, f_{\pm 1}(\x) = \pm 2x_1 + 3 x_2, f_{\pm 2}(\x)= \pm 5x_1 + 2 x_2.
\]
The plot below shows  that part of the level curve $[f=a]$ which lies in the upper
half plane $x_2 \geq 0$. It consists of the polygon connecting
the five points $(-\frac{a}{5},0)$, $(-\frac{a}{11},\frac{3a}{11})$,
$(0,\frac{a}{3})$, $(\frac{a}{11},\frac{3a}{11})$, $(\frac{a}{5},0)$.
We are interested in that part of the lower level set
$[f\leq a]$, which lies within the gray-shaded dragon-shaped area inside
the polygon $[f\leq a]$, and above the $x_1$-axis.

\centerline{
\includegraphics[scale=0.8]{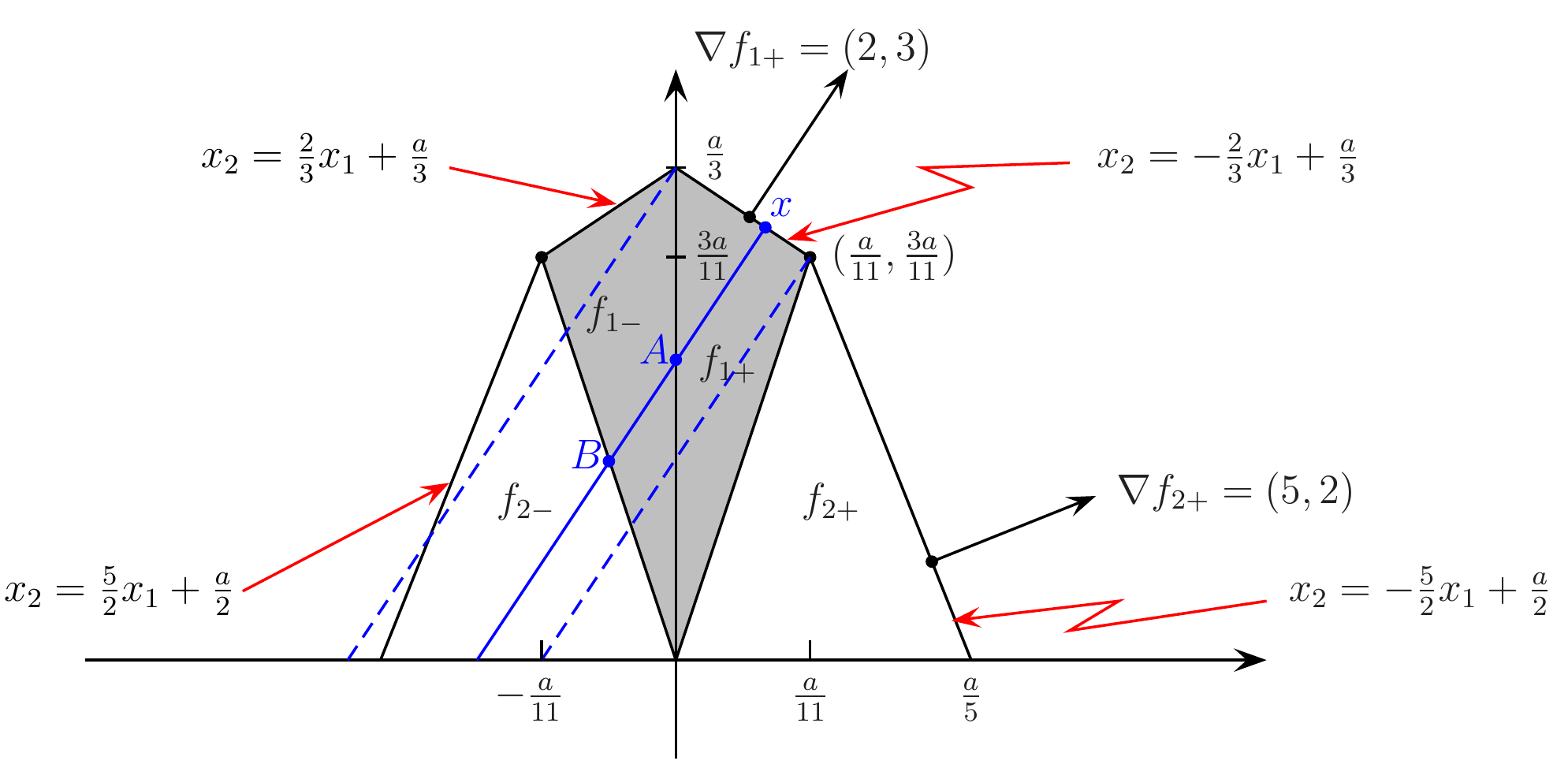}
}

Consider the exceptional
set $N = \cup_{i\not=j} \{f_i = f_j = f\}$, whose intersection with
the upper half-plane $x_2\geq 0$ consists of the three lines
$x_1=0$, $x_2=\pm 3x_1$. Then for $\x\not\in N$ the gradient
$\nabla f(\x)$ is unique.  We will generate a sequence $\x^j$ of iterates
which never meets $N$, so that $\phi^\sharp(\y,\x) = f(\x) + \nabla f(\x)^\top(\y-\x)$
with $\nabla f(\x)\in \{\pm(2,3),\pm (5,2)\}$ at all iterates $\x^j$. 
It will turn out that serious iterates
$\x^j$ never leave the dragon area, only trial points may. 

Assume that our current iterate $\x$ has $f(\x)=a$ and is situated on the right upper
part of the $a$-dragon, shown as the blue $x$ in the figure. That means
\[
\x = (x_1,-\textstyle\frac{2}{3}x_1+\frac{a}{3}), \quad f(\x)=a, \quad 0 < x_1 \leq \frac{a}{11}.
\]
Then $\phi^\sharp(\y,\x) = f_{1+}(\y) = 2y_1+3y_2$. If the current trust-region radius
is $R = \sqrt{13}r$, then the solution of (\ref{tangent}) is $\y=\x + r(-2,-3)=(x_1-2r,-\frac{2}{3}x_1+\frac{a}{3}-3r)$.
If we follow the point $y$ as a function of $r$ along the steepest descent
line shown in blue, we will reach the points $A,B$ in increasing order at
$0<r_A < r_B$. Here 
$A$ is the intersection of the steepest descent line with the $x_2$ axis,
reached at $r_A=x_1/2$. The point  $B$ is when the ray meets the boundary of the $a$-dragon, 
which is  the line $x_2=-3x_1$ on the left,   reached at
$$
r_B= \textstyle\frac{7}{27}x_1+\frac{a}{27}.
$$ 
We have $f(A)=f_{1+}(A)=a-\frac{17}{4}x_1$ and $f(B)=f_{1-}(B)=-\frac{143}{27}x_1+\frac{22}{2}a$,
and from here on $f$ increases along the ray.
The test quotient $\rho$ for trial points $\y$ of this form behaves as follows
\[
\rho = \frac{f(\x_a)-f(\y)}{f(\x_a)-\phi^\sharp(\y,\x_a)}
=
\left\{
\begin{array}{ll}
1 & \mbox{ if } 0 < r \leq r_A \\
\frac{4x_1+5r}{13r} & \mbox{ if } r_A \leq r \leq r_B\\
\frac{a-12r+19x_1}{39r} & \mbox{ if } r_B \leq r <\infty
\end{array}
\right.
\]
The quotient is therefore constant on $[0,r_A]$, and decreasing on $[r_A,\infty)$.
If we trace the quotient at the point
$B$ as a function of $x_1$, we see that
$\rho = \frac{5}{13}$ at $x_1=0$, and $\rho = \frac{198}{234}$  at $x_1=\frac{a}{11}$. That means if we take the Armijo
constant as $\gamma \in (\frac{198}{234},1)$, then none of the points in $[B,\infty)$ is accepted, whatever
$x_1 \in (0,\frac{a}{11}]$. 
Let  the value $r$ where the quotient $\rho$ equals $\gamma$
be called $r_\gamma$. Then $r_A < r_\gamma < r_B$, and we have
$
r_\gamma = \frac{4x_1}{13\gamma -5}.
$

Let us for simplicity put $\Gamma = 1$. That means good steps where the trust-region
radius is doubled  are exactly those in $(x,A]$, that is,  $0 < r \leq r_A$. 
Such a step is immediately accepted, 
and we stay on the right upper half of the $a^+$-dragon, where $a^+ < a$, except for the point $A$, which we will exclude later.
We find for $0 < r < r_A=x_1/2$:
\[
a^+=a-13r >0, \quad x^+ =\textstyle (x_1-2r, -\frac{2}{3}x_1+\frac{a}{3}-3r) = (x_1^+,-\frac{2}{3}x_1^++\frac{a^+}{3}).
\]
Note that $a=a^+$ for the limiting case $x_1=0$, and $a^+=\frac{9}{22}a$ for the limiting case $x_1=\frac{a}{11}$.
According to step 8 of the algorithm the trust-region radius is doubled $(R^+ =2 R)$
for $0 < r < r_A$, because $\rho = 1 \geq \Gamma = 1$.

The second case is when from the current $x$ with $f(x)=a$ a step with $R=\sqrt{13}r$ and $r\in (r_A,r_\gamma)$
is taken. Then we end up on the left hand side of the dragon with the new situation
\[
x^+=\textstyle (x_1-2r,-\frac{2}{3}x_1+\frac{a}{3}-3r), \quad f(x^+)=f_{1-}(x^+)
=-4x_1+a-5r = a^+.
\] 
By symmetry, this case
is analogous to the initial situation, the model at $\x^+$ now being $f_{1-}$. We are now on the upper left side of the smaller $a^+$-dragon.
Since $\gamma \leq \rho < \Gamma$, the trust-region
radius remains unchanged.

The third case is when $r \in [r_\gamma,\infty)$. Here the step is rejected,
and the trust-region radius is halved, until a value $r<r_\gamma$ is reached. 

Since
$\phi^\sharp$ is used, no cutting planes are taken, and we follow the classical trust-region method.
In consequence, the serious iterates $x,x^+,x^{++},\dots$ stay in the dragons $a,a^+,a^{++},\dots$ and converge to the origin, which is not a 
critical point of $f$. Note that we have to assure that none of the trial points $\y$ lies precisely on the
$x_2$-axis. Now it is clear that for a given starting point
$x$ the method has a countable number of possible trial steps $\y^k$, and we can choose the initial
$x_1\in (0,\frac{a}{11}]$ such that the $x_2$-axis is avoided, for instance, by taking an irrational initial value. Alternatively, in the case where
$\y^k$ hits the $x_2$-axis, we might use rule (\ref{trial}) to change it slightly
to a $\z^k$, which is not on the axis. In both cases the method will never leave the dragon area,
hence convergence based on the Cauchy point fails.

\section{Parametric robustness}
\label{sect_experiments} 
We consider an  LFT plant  \cite{zhou}  with real parametric uncertainties $\mathcal F_u(P, \Delta)$,  where 
\begin{eqnarray}\label{plant}
P(s):\left\{
\begin{matrix}
\dot{x}& =& Ax &+& B_pp   &+&B_w w&  \\
q& =& C_qx&+&D_{qp}p&+&D_{qw}w& \\
z& =& C_zx& +&D_{zp}p & +&D_{zw}w & 
\end{matrix}
\right.
\end{eqnarray}
and $x\in \mathbb R^{n_x}$ is the state, 
$w \in \mathbb R^{m_1}$ the vector of exogenous inputs, 
and $z\in \mathbb R^{p_1}$ the regulated output.  
The uncertainty channel is defined as $p = \Delta q$, 
where the uncertain matrix
$\Delta$  is without loss assumed to have the block-diagonal  form
\begin{equation}
\label{matrix}
\Delta = {\rm diag}\left[ \delta_1 I_{r_1},\dots, \delta_m I_{r_m}\right]
\end{equation}
with $\delta_1,\dots,\delta_m$ representing real uncertain
parameters, and $r_i$ giving the number of repetitions of $\delta_i$. 
We write $\del=(\delta_1,\dots,\delta_m)$
and assume without loss that $\del = 0$  represents the nominal parameter value.
Moreover, we consider $\del\in \mathbb R^m$ in one-to-one correspondence with
the matrix $\Delta$ in (\ref{matrix}).

\subsection{Worst case $H_\infty$-performance over a parameter set}
\label{sec-h}
Our first problem 
concerns analysis of the performance of a system
(\ref{plant}) subject to parametric uncertainty.  In order to analyze the robustness of (\ref{plant})
we  compute the worst-case $H_\infty$ performance of the channel $w\to z$ 
over a given uncertain  parameter range normalized to
${\bf \Delta}=[-1,1]^m$.  In other words, we compute
\begin{equation}
\label{h}
h^*= \max\{ \|T_{wz}(\del)\|_\infty: \del\in {\bf \Delta}\},
\end{equation}
where $T_{wz}(\delta)$ is the transfer function $z(s) = \mathcal F_u(P(s),\Delta) w(s)$, or more explicitly,
\[
z(s)=\left[ P_{22}(s) + P_{21}(s) \Delta (I - P_{11}(s) \Delta)^{-1} P_{12}(s) \right] w(s) .
\]
The significance of (\ref{h}) is
that computing a critical parameter value $\del^*\in {\bf \Delta}$ which degrades the $H_\infty$-performance
of (\ref{plant}) may be an important domino in assessing the properties of a controlled system (\ref{plant}). We refer to
\cite{and} where this is exploited in parametric robust synthesis.

Solving (\ref{h})
leads to a program of the form (\ref{program}) if we write (\ref{h})  as minimization of
$h_-(\del)=-\|T_{wz}(\del)\|_\infty$ over the convex ${\bf \Delta}$.  The specific form of
${\bf \Delta}$ strongly suggest the use of the maximum norm
$|\del|_\infty=\max\{ |\delta_1|,\dots,|\delta_m| \}$ to define trust-regions.
Moreover, we will use the standard model $\phi^\sharp$
of $h_-(\del)=-\|T_{wz}(\del)\|_\infty$, as is
justified by the following

\begin{lemma}
Let $D=\{\del: T_{zw}(\del) \text{ \rm  is internally stable}   \}$. Then   $h_-: \del \mapsto- \|T_{zw}(\del)\|_\infty$
is upper-$C^1$ on $D$.
\end{lemma}

\begin{proof}
It suffices to prove that $h_+:\del \mapsto \|T_{wz}(\del)\|_\infty$ is lower $C^1$. To prove this,
recall that the maximum singular value has the variational
representation
\[
\overline{\sigma}(G) = \sup_{\|u\|=1} \sup_{ \|v\|=1} \left| u^T G v\right|.
\]
Now observe that $z\mapsto |z|$, being convex,  is lower-$C^1$ as a mapping $\mathbb R^2 \to \mathbb R$, so we may write
it as
\[
|z| = \sup_{l\in \mathbb L} \Psi(z,l)
\]
for $\Psi$ jointly of class $C^1$ and a suitable compact set $\mathbb L$. Then
\begin{equation}
\label{4sup}
h_+(\delta) = \sup_{j\omega \in \mathbb S^1} \sup_{\|u\|=1} \sup_{ \|v\|=1} \sup_{l\in\mathbb L}
\Psi\left( u^TT_{zw}(\delta,j\omega)v, l  \right),
\end{equation}
where $\mathbb S^1 = \{j\omega: \omega\in \mathbb R \cup\{\infty\}\}$ is homeomorphic with the $1$-sphere.
This is  a representation of the form (\ref{lower}) for $h_+$, where the compact space 
is  $\mathbb K:=\mathbb S^1 \times \{u: \|u\|=1\}\times \{v: \|v\|=1\} \times\mathbb L$, $F$ is 
$F(\delta,j\omega,u,v,l):= \Psi\left( u^TT_{zw}(\delta,j\omega)v, l  \right)$ and $y= (j\omega,u,v,l)$.  
\end{proof}

\begin{theorem}[Worst-case $H_\infty$ norm on $\bf \Delta$]
\label{theorem_h}  
Let $\del^j \in {\bf \Delta}$ be the sequence generated by the standard trust-region algorithm
applied to  program {\rm (\ref{h})} based on the standard model of $h_-$. 
Then the  $\del^j$ converge to a 
critical point $\del^*$ of {\rm (\ref{h})}.
\end{theorem}

\begin{proof}
By Lemma \ref{upper} 
Algorithm \ref{algo1}  coincides with a classical first-order trust-region algorithm, with
convergence in the sense of subsequences. 
Convergence to a single critical point
then follows by observing that $h_-$ satisfies a \L ojasiewicz inequality.
\end{proof}

\subsection{Robust stability over a parameter set}
\label{sec-alpha}
In our
second problem  we wish to check whether the uncertain system (\ref{plant}) is robustly stable over the uncertain
parameter set ${\bf\Delta}=[-1,1]^m$. This can be tested
by maximizing the spectral abscissa
over ${\bf \Delta}$:
\begin{equation}
\label{alpha}
\alpha^* = \max\{ \alpha\left( A(\del)\right): \del\in {\bf \Delta}\},
\end{equation}
where $A(\del)$ is the closed-loop system matrix
\begin{equation}
\label{AofDelta}
A(\del) = A + B_p\Delta \left( I-D_{qp}\Delta \right)^{-1}C_q, 
\end{equation}
and where the spectral abscissa of $A\in \mathbb R^{n\times n}$ is 
$\alpha(A)=\max\{{\rm Re}(\lambda): \lambda \mbox{ eigenvalue of }A\}$. 
The decision is now as follows. 
As soon as $\alpha^*\geq 0$, the solution 
$\del^*$ of (\ref{alpha}) represents a destabilizing choice of the parameters,
and this may be valuable information in practice, see  \cite{and}. On the other hand, if the 
global maximum has value $\alpha^* < 0$, then a certificate for robust stability
over $\delta\in {\bf \Delta}$ is obtained. 

Global maximization of (\ref{alpha}) is known to be NP-hard \cite{poljak,braatz}, 
so it is interesting to use a local optimization method to compute good lower bounds.
This can be achieved by algorithm \ref{algo1}, because
(\ref{alpha}) is clearly of the form (\ref{program})
if maximization of $\alpha$ is replaced by minimization of $-\alpha$ over ${\bf \Delta}$.
In our experiment  additional speed is gained by adapting the trust-region norm $|\del|_\infty
=\max\{|\delta_1|,\dots,|\delta_m|\}$ to the special form ${\bf \Delta}=[-1,1]^m$
of the set $C$, and  the standard model  $\phi^\sharp$
of $a_-(\del)=-\alpha(A(\del))$ is used.
With these arrangements the method converges fast and reliably to a local optimum,
which in the majority of cases can be certified {\em a posteriori} as a global one. 

In order to justify the use of the standard model in Algorithm \ref{algo1} we have to show
that $a_-$ is upper-$C^1$, or  at least that its standard model is strict. Here the situation is  more delicate
than in section \ref{sec-h}. We start by observing the following

\begin{lemma}
\label{prop4}
Suppose all active eigenvalues  of $A(\del)$ at $\del$ are  semi-simple. Then $a_-(\del)= -\alpha\left( A(\del)\right)$
is Clarke subdifferentiable in a neighborhood of $\del$. 
\end{lemma}

\begin{proof}
This  follows from \cite{BO94}. A very concise proof that semi-simple
eigenvalue functions are locally Lipschitz could also be found in \cite{Lui11}. 
\end{proof}

That $a_\pm(\delta) = \pm \alpha (A(\delta))$ may fail to be locally Lipschitz
was first observed in \cite{BO94}. 
This may lead to difficulties when $a_+$  is minimized. In contrast, in
our numerical testing it is $a_-(\del)=-\alpha\left( A(\del)\right)$ which is minimized, and this
behaves consistently like an upper-$C^1$ function. Theoretically
we expect $a_-$ to have a strict standard model if all active eigenvalues
of $A(\del^*)$ are semi-simple.  
An argument indicating that its standard model is at least directionally strict is given
in \cite[V.C]{and}. See  \cite{MBO97} for more information on $a_\pm$.

\begin{theorem}[Worst-case spectral abscissa on $\bf \Delta$]
\label{theorem2} 
Let $\del^j\in {\bf \Delta}$ be the sequence generated by Algorithm
{\rm \ref{algo1}} for program {\rm (\ref{alpha})}, where
the standard model $\phi^\sharp$ of $a_-$ is used.
Suppose every accumulation point $\del^*$ of the sequence $\del^j$ is simple. Then  the sequence $\delta^j$ converges  
to a critical point of  {\rm (\ref{alpha})}.
\end{theorem}

\begin{proof}
We apply Theorem \ref{theorem1} to get convergence in the sense of subsequences. 
\end{proof}

\subsection{Distance to instability}
Our third problem is related to the above and
concerns computation of the structured distance to instability of (\ref{plant}).
Suppose the matrix $A$ in (\ref{plant}) is nominally stable, i.e., $A(\delta)$ is stable at the nominal
$\del = 0$.  Then the structured distance to instability is defined
as
\begin{equation}
\label{dist}
d^* = \max\{ d > 0:  A(\del) \mbox{ stable for all } |\del|_\infty  < d\},
\end{equation}
where $A(\del)$ is given by (\ref{AofDelta}), and
$|\del|_\infty=\max\{|\delta_1],\dots,|\delta_m|\}$.
Equivalently, we may consider
the following constrained optimization program
\begin{eqnarray}
\label{dist2}
\begin{array}{ll}
\mbox{minimize} & t \\
\mbox{subject to}& -t \leq \delta_i \leq t \\
&\alpha\left( A(\del) \right)\geq 0
\end{array}
\end{eqnarray}
with decision variable $\x=(t,\del)\in \mathbb R^{m+1}$. Introducing
the convex set
$C=\{(t,\del): -t\leq \delta_i\leq t, i=1,\dots,m\}$, this can be 
transformed to program (\ref{program})
if we minimize an exact penalty objective $f(\x)=t + c \max\left\{0,-\alpha\left( A(\del) \right)\right\}$
with a penalty constant $c>0$ over $C$. 

It is clear that the objective of $f$ has essentially the same properties as $a_-$.
It suffices to argue that $\partial \max\{0,-\alpha(A(\del))\} = {\rm co} \{0\} \cup \partial a_-(\del)$
at points $\del$ where $a_-$ is locally Lipschitz and $a_-(\del)=0$. Indeed,
the inclusion $\subset$ holds in general. For the reverse inclusion it suffices to
observe that $0\in \partial \max\{0,-\alpha(A(\del))\}$ for those $\del$ where $a_-(\del)=0$. This is clear,
because $0$ is a minorant of this max function. We may then use the following

\begin{lemma}
Suppose $f = \max\{f_1,f_2\}$ and $f_i$ has a strict model $\phi_i$. Then $\phi = \max\{\phi_1,\phi_2\}$
is a strict model of $f$ at those $\x$ where $\partial f(\x)={\rm co} \left(\partial f_1(\x) \cup \partial f_2(\x)\right)$.
\end{lemma}

\begin{proof}
In fact, the only axiom which does not follow immediately is $(M_1)$.
We only know $\partial_1 \phi_i(\x,\x) \subset \partial f_i(\x)$, so
$\partial_1 \phi(\x,\x) = {\rm co} \left( \partial_1 \phi_1(\x,\x)\cup \partial_1 \phi_2(\x,\x) \right)
\subset {\rm co} \left(\partial f_1(\x) \cup \partial f_2(\x)\right)$. For those $\x$ where the
maximum rule is exact, this implies indeed $\partial_1\phi(\x,\x) \subset \partial f(\x)$.
\end{proof}

This means that we can use the model
$\phi(\del',t',\del,t) =t'+c \max\{0,\phi^\sharp(\del',\del)\}$ in
Algorithm \ref{algo1} to solve (\ref{dist2}), naturally with the same proviso as in section
\ref{sec-alpha}, where we need the standard model $\phi^\sharp$ of $a_-$ to be strict.

\begin{center}
\begin{table}[h!]
{\small
\begin{center}
\begin{tabular}{|| c | l | c | l | c | c | c | c | c | c ||}
\hline \hline
$\sharp$ & {\bf Benchmark} & $n$ &  {\bf Structure} & $\underline{h}$ &  $h^*$ & $\overline{h}$ &$t^*$ & ${\overline{h}}/{h^*}$ & ${t_{\rm wc}}/{t^*}$\\
\hline\hline
1 &Beam1 & 11 &   $1^3 3^1 1^1$& 1.70&1.71 &1.70& 1.02 & 0.99 & 13.29 \\
\hline
2 & Beam2 & 11 &  $1^3 3^1 1^1$ &1.29  &     1.29                        &1.29 & 0.36 & 1       & 32.68            \\
\hline
3 & DC motor 1 & 7 &  $1^12^2$                   &0.72 &0.72 &0.72&  0.51 & 1.01 &14.49      \\
\hline
4 & DC motor 2 & 7 &       $1^12^2$     &0.50 &  0.50                   &0.50 &0.13 &  1       &45.02    \\
\hline
5 & DVD driver 1 & 10 &   $1^13^31^13^1$                 & 45.45 &45.45& 45.46&0.23& 1      &189.31 \\
\hline
6 & Four-disk system 1 & 16 &  $1^13^51^4$                 & 3.50 &4.56 &3.50 &0.44 & 0.77&343.35 \\
\hline
7 & Four-disk  system 2 & 16 &      $1^13^51^4$           & 0.69 &0.68 &0.69&0.34  & 1.01 & 558.03\\
\hline
8&Four-tank system 1 &12 &         $1^4$      & 5.60 &5.60 &5.60&0.32  &  1      & 5.72 \\
\hline
9 & Four-tank system 2 &12 & $1^4$   & 5.60&  5.57 &5.60&0.29  & 1       & 7.32 \\
\hline
10&Hard disk driver 1 & 22 & $1^32^41^4$     & 243.9 &7526.6& Inf &0.96& Inf  & 73.10\\
\hline
11&Hard disk driver 2 & 22&  $1^32^41^4$    &0.03 &0.03  &0.03 & 0.20  & 1.12 & 314.92\\
\hline
12&Hydraulic servo 1 & 9 &       $1^9$            &1.17 &1.17  &1.17& 0.34  &  1      & 10.94 \\
\hline
13&Hydraulic servo 2 & 9 & $1^9$           &0.7 & 0.70  &0.7 & 0.33  & 1.01  & 11.69 \\
\hline
14 &Mass-spring 1 & 8 &   $1^2$                & 3.71 &6.19 &3.71&0.31  & 0.60  & 3.54  \\
\hline
15 & Mass-spring 2 & 8 &   $1^2$    &6.84 & 6.84         & 7.16  &0.13  & 1.05  & 7.05  \\
\hline
16& Missile 1 & 35 &  $1^36^3$                 & 5.12 &5.15 &5.12 &0.46 & 0.99   & 272.54\\
\hline
17 & Missile 2 & 35 &   $1^36^3$       &1.83 & 1.82        & 1.83  &0.22  & 1        &1183.5 \\
\hline
18 & Filter 1 & 8 &  $1^1$           &4.86 & 4.86 &4.86& 0.32 & 1         & 3.41\\
\hline
19 & Filter 2 & 3 & $1^1$      &2.63 &   2.64       & 2.63 &0.27  & 1         & 4.06 \\
\hline
20& Filter-Kim 1 & 3 &  $1^2$        &2.95 &2.96 &2.95 & 0.24  & 1        & 3.4\\
\hline
21 &Filter-Kim 2 & 3 & $1^2$ &2.79  & 2.79       &2.79 & 0.07  &  1       & 12.95 \\
\hline
22 & Satellite 1 & 11 &        $1^16^11^1$                       &0.16 &0.17&0.16  &0.33  & 1        & 86.17\\
\hline
23 & Satellite 2 & 11 &   $1^16^11^1$   & 0.15 & 0.15                        & 0.15  &0.70  & 1        & 41.09\\
\hline
24 & Mass-spring-damper  1 & 13 & $1^1$&7.63 &8.85 &7.63&0.21  &0.86   &4.88 \\
\hline
25 & Mass-spring-damper 2 & 13 & $1^1$&1.65 &1.65 &1.65 &0.08  & 1       &13.70\\
\hline
26 & Robust Toy 1 &3 & $1^12^1$&0.12 &0.12 &0.12&0.56 &1 &4.24 \\
\hline
27 & Robust Toy 2 &3 &$1^22^23^1$ &20.85 &21.70 &20.91 &0.24 &0.96 &29.19\\
\hline\hline
\end{tabular}
\end{center}
}
\vspace*{.05cm}
\caption{Benchmarks  for worst-case $H_\infty$-norm on ${\bf \Delta}$ \label{table1}}
\end{table}
\end{center}

\section{Experiments}
In this part experiments with algorithm \ref{algo1}
applied to programs (\ref{h}), (\ref{alpha}) and  (\ref{dist}) are reported.

\subsection{Worst-case $H_\infty$-norm}
We apply algorithm \ref{algo1}  to program (\ref{h}). 
Table \ref{table1} shows the result for 27 benchmark systems, where
$n$ is the number of states, and column 4 gives the uncertain structure
$[r_1 \dots r_m]$ according to (\ref{matrix}). An  expression like $1^33^11^1$ corresponds to $
[r_1\,r_2\,r_3\,r_4\,r_5]=[1\, 1\, 1\, 3\, 1]$. 
The values achieved by algorithm \ref{algo1}
are $h^*$ in column 6, obtained in $t^*$ seconds  CPU. 
To certify  $h^*$  we 
use the function {\tt WCGAIN} of \cite{wcgain}, which is a
branch-and-bound method tailored to program (\ref{h}).
{\tt WCGAIN} computes a lower and an upper bound $\underline{h}, \overline{h}$ shown in columns 5,7
within $t_{\rm wc}$ seconds. It also
provides a $\underline{\delta}\in {\bf \Delta}$  realizing the lower bound.

The results in table \ref{table1} show that $h^*$ is certified by
{\tt WCGAIN}  in the majority of cases
1-5,7-9,11-13,16,17. Case 15 leaves a doubt, while cases 6,10,14,24 are failures of {\tt WCGAIN}.
On average algorithm \ref{algo1} was 121-times faster than {\tt WCGAIN}. 
The fact that both methods are in good agreement can be understood as an endorsement of
our approach.

\subsection{Robust stability over ${\bf \Delta}$}
In our second test algorithm \ref{algo1}  is applied to program
(\ref{alpha}). We have used a bench of 32 cases 
gathered in Table \ref{table2}, and algorithm \ref{algo1} converges to the value $\alpha^*$
in $t^*$ seconds.  
\begin{table}[!h]
{\small
\begin{tabular}{|| c | l | c | l | c | c | c | c ||}
\hline\hline
$\sharp$  & {\bf Benchmark} & $n$ & {\bf Structure} & $\alpha^*$ & $\alpha_{\rm ZM}$ & $t^*$ &$ t_{\rm ZM}$ \\
\hline \hline
28 & Beam3    & 11 &  $1^33^11^1$&-1.2e-7 & -1.2e-7 & 0.19 & 32.70 \\
\hline
29 & Beam4 & 11 & $1^33^11^1$ &-1.7e-7 & -1.7e-7 & 0.04 & 33.00 \\
\hline
30 & Dashpot system 1 & 17 &  $1^6$& 0.0186 & 0.0185 & 0.23 & 90.25 \\
\hline
31 & Dashpot system 2 & 17 & $1^6$ &-1.0e-6 & -1.0e-6 & 0.39 & 39.63 \\
\hline
32 & Dashpot system 3 & 17 & $1^6$& -1.6e-6 & -1.6e-6 & 0.08 & 39.70\\
\hline
33 & DC motor 3 & 7 & $1^12^2$  & -0.0010 & -0.0010 & 0.02 & 20.63\\
\hline
34 & DC motor 4& 7 & $1^12^2$& -0.0010 & -0.0010 & 0.02 & 20.74\\
\hline
35 & DVD driver 2 & 10 &  $1^13^31^13^1$& -0.0165 & -0.0165 & 0.04 & 49.29\\
\hline
36 & Four disk system 3 & 16 &$1^13^51^4$ & 0.0089 & 0.0088 & 0.10 & 159.61\\
\hline
37 & Four disk system 4 & 16 &$1^13^51^4$& -7.5e-7 & -7.5e-7 & 0.29 & 73.86 \\
\hline
38 & Four disk system 5 & 16 & $1^13^51^4$& -7.5e-7 & -7.5e-7 & 0.29 & 74.36 \\
\hline
39 & Four tank system 3 & 12 & $1^4$& -6.0e-6 & -6.0e-6 & 0.17 & 25.81 \\
\hline
40 & Four tank system 4 & 12 &$1^4$  & -6.0e-6 & -6.0e-6 & 0.02 & 26.20\\
\hline
41 & Hard disk driver 3 & 22 &  $1^32^41^4$&  266.70 & 266.70 & 0.09 & 297.21\\
\hline
42 & Hard disk driver 4 & 22 & $1^32^41^4$& -1.6026 & -1.6026 & 0.06 & 80.40\\
\hline
43 & Hydraulic servo 3 & 9 & $1^9$& -0.3000 & -0.3000 & 0.04 & 51.41 \\
\hline
44 & Hydraulic servo 4 & 9 & $1^9$& -0.3000 & -0.3000 & 0.02 & 50.95 \\
\hline
45 & Mass-spring 3 & 8 & $1^2$& -0.0054 & -0.0054 & 0.01 & 31.59\\
\hline
46 & Mass-spring 4 & 8 &$1^2$& -0.0368 & -0.0370 & 0.01 & 16.94 \\
\hline
47 & Missile 3 & 35 &$1^36^3$& 22.6302 & 22.1682 & 0.07 & 104.18 \\
\hline
48 & Missile 4 & 35 &$1^36^3$& -0.5000 & -0.5000 & 0.07 & 51.78 \\
\hline
49 & Missile 5 & 35 &$1^36^3$& -0.5000 & -0.5000 & 0.07 & 52.24\\
\hline
50 & Filter 3 & 8 & $1^1$& -0.0148 & -0.0148 & 0.06 & 7.05\\
\hline
51 & Filter 4 & 8 & $1^1$& -0.0148 & -0.0148 & 0.02 & 6.89\\
\hline
52 & Filter-Kim 3 & 3 & $1^2$& -0.2500 & -0.2500 & 0.01 & 12.83 \\
\hline
53 & Filter-Kim 4 & 3 & $1^2$ & -0.2500 & -0.2500 & 0.01 & 12.90\\
\hline
54 & Satellite 3 & 11 & $1^16^11^1$& 3.9e-5 &  3.9e-5 & 0.02 & 44.02 \\
\hline
55 & Satellite 4 & 11 &$1^16^11^1$& -0.0269 & -0.0269 & 0.02 & 26.02\\
\hline
56 & Satellite 5& 11 & $1^16^11^1$& -0.0268 & -0.0268 &  0.02 & 26.08\\
\hline
57 & Mass-spring-damper 3 & 13 & $1^1$ & 0.2022 & 0.2022 & 0.01 & 8.30\\
\hline
58 & Mass-spring-damper 4 & 13 & $1^1$ & -0.1000 & -0.1000 & 0.01 & 6.91 \\
\hline
59 & Mass-spring-damper 5 & 13 & $1^1$& -0.1000 & -0.1000 & 0.01 & 6.94\\
\hline\hline
\end{tabular}
}
\vspace*{0.05cm}
\caption{Benchmarks  for worst-case spectral abscissa (\ref{alpha}). \label{table2}}
\end{table}
To certify $\alpha^*$ we have implemented algorithm 
 \ref{algo2}, known as
integral global optimization, or as the Zheng-method (ZM), based on \cite{ZM}. 
\begin{algorithm}
\caption{Zheng-method for global optimization  $\alpha^* = \max_{\x\in {\bf \Delta}} f(\x)$}
\label{algo2}
\noindent\fbox{%
\begin{minipage}[b]{\dimexpr\textwidth-\algorithmicindent\relax}
\begin{algorithmic}
\STEP{Initialize} 
	Choose initial $\alpha < \alpha^*$.
	\STEP{Iterate}
	Compute $\alpha^+ =\displaystyle \frac{\int_{[f\geq \alpha]} f(x)\, d\mu(x)}{\mu [ f\geq \alpha]}$.
	\STEP{Stopping} If progress of $\alpha^+$ over $\alpha$ is marginal, stop, otherwise update $\alpha$
	by $\alpha^+$ and loop on with step 2.
\end{algorithmic}
\end{minipage}
}%
\end{algorithm}
\noindent
Here $\mu$ is any continuous finite Borel measure on ${\bf \Delta}$. Numerical implementations
use Monte-Carlo to compute the integral, and we refer to \cite{ZM} for details. Our numerical tests
are performed with $2000\cdot m$ initial samples, and stopping criterion variance$\,\, =10^{-7}$; cf. \cite{ZM}
for details. The result obtained by
ZM are $\alpha_{\rm ZM}$ obtained in $t_{\rm ZM}$ seconds CPU.

A favorable  feature of ZM is that it can  be initialized with the
lower bound  $\alpha^*$, and this leads to a significant
speedup. Altogether ZM and algorithm \ref{algo1} are in very
good agreement on the test bench, which we consider an argument in favor
of our approach.

\begin{table}[!h]
{\small
\begin{center}
\begin{tabular}{|| c | c | c | c |  c | c || c | c || c    ||}
\hline\hline
$\sharp$ & {\bf Benchmark} & $n$ & {\bf Structure} &$d^*$ & $d_{\rm F}/d^*$& ${\bf D_{\rm ZM}}$& $t^*$ &$t_{\rm ZM}$    \\
\hline\hline
60 &Academic example  & 5 &   $1^1$   &    0.79 &    1  &    $\surd$  &    0.15 & 7.3      \\ \hline 
    61 &Academic example  & 4 &   $1^3$ &    3.41 &    1  &   $\surd$  &    0.13& 23.9      \\ \hline 
    62 & Academic example & 4 &          $2^2$  &    0.58 &    1  &    $\surd$ &    0.15&97.4     \\ \hline  
    63 & Inverted pendulum & 4 &     $1^3$ &    0.84 &    1  &    $\surd$ &    0.22 &24.7    \\ \hline  
    64 & DC motor & 4 &    $1^3$$2^1$  $1^1$      &    1.25 &    1  &   $\surd$  &    0.19&37.7    \\ \hline  
    65 & Bus steering system & 9 &    $2^1$$3^1$  &    1.32 &    0.99 &   $\surd$  &    0.37& 13.8     \\ \hline 
    66 & Satellite & 9 &   $2^11^2$  &    1.01 &    0.99 &    $\surd$  &    0.3 & 20.2    \\ \hline
    67&Bank-to-turn missile &6 &  $1^4$      &    0.60 &    0.99 &   $\surd$ &    0.17& 167.7    \\ \hline
   68 & Aeronautical vehicle &8 &  $1^4$ &    0.61 &    0.99 &   $\surd$  &    0.19&38.9    \\ \hline
   69 &Four-tank system & 10 &    $1^4$  &    6.67 &    0.99 &    $\surd$  &    0.27& 24.9    \\ \hline
   70  &Re-entry vehicle & 6&      $3^12^13^1$   &    6.20 &    1  &   $\surd$  &    0.44& 21.8    \\ \hline
   71  &Missile & 14 &   $1^4$  &    7.99 &    1  &    $\surd$  &    0.25& 24.9    \\ \hline
   72 &Cassini spacecraft & 17 &      $1^4$ &    0.06 &    1  &   $\surd$  &    0.13&25.1     \\ \hline
   73 &Mass-spring-damper & 7 &      $1^6$        &    1.17 &    1  &    $\surd$  &    0.17& 2536.3     \\ \hline
   74 & Spark ignition engine & 4 &    $1^7$  &    1.22 &    0.99 &   $\surd$  &    0.41&42.8      \\ \hline
   75& Hydraulic servo system & 8 &   $1^8$  &    1.50 &    0.99 &    $\surd$  &    0.41& 62.8    \\ \hline
   76  & Academic example & 41 &   $2^11^3$ &    1.18 &    0.99 &    $\surd$  &    0.57&36.5     \\ \hline
   77 & Drive-by-wire vehicle & 4 &     $1^2$$2^7$ &    1  &    0.99 &    $\surd$ &    0.96& 97.0   \\ \hline
   78   & Re-entry vehicle & 7 &     $1^3$$6^1$$4^1$&    1.02 &    0.98 &    $\surd$ &    0.42&132.4      \\ \hline
   79  & Space shuttle& 34 &     $1^9$&    0.79 &    0.99 &  $\surd$  &    0.8& 60.9    \\ \hline
   80   &Rigid aircraft& 9 &  $1^{14}$ &    5.42 &    1  &    $\surd$  &    0.54& 252.5     \\ \hline
   81 & Fighter aircraft & 10 &     $3^115^11^62^11^1$ &    0.59 &    0.99 &    $\surd$  &    1.31& 171.3   \\ \hline
   82 & Flexible aircraft & 46 &  $1^{20}$  &    0.22 &    0.99 &   $\surd$ &    1.26&180.3   \\ \hline
   83  & Telescope mockup & 70 & $1^{20}$  &    0.02 &    0.99 &    $\surd$  &    1.37&274.8   \\ \hline
   84   & Hard disk drive & 29 &     $1^82^41^11$&    0.82 &    1  &    $\surd$ &    2.87& 202.1    \\ \hline
   85  & Launcher &30 &    $1^22^21^23^16^11^{12}2^8$ &    1.16 &    0.99 &   $\surd$ &    4.08&271.2     \\ \hline
   86  & Helicopter &12 &  $30^4$  &    0.08 &    0.99 &    $\surd$  &    0.85&70.7     \\ \hline
87& Biochemical network &7 &  $39^{13}$ & 0.00 & 1           & failed & 36.76 & - \\
\hline\hline
\end{tabular}
\end{center}
\vspace*{0.05cm}
\caption{Benchmarks  for  distance to instability (\ref{dist}), available in \cite{smac}.  \label{table3}}
}
\end{table}

\subsection{Distance to instability}
In this last part we apply Algorithm \ref{algo1} to (\ref{dist}) using the test bench of Table \ref{table3}, which can be found in \cite{fabrizi}.
The distance computed by Algorithm \ref{algo1} is $d^*$ in column 2 of Table \ref{table3}.
We certify $d^*$  
using ZM \cite{ZM} and by comparing to the  local method of \cite{fabrizi}. 

To begin with, ZM is used in the following way.
For a given $d^*$  and 
a confidence level $\gamma = 0.05$ we compute
\begin{equation}
\label{decision1}
\underline{\alpha} = \max\{ \alpha(A(\delta)): \delta \in (1-\gamma)d^* {\bf \Delta}\}
\end{equation}
and
\begin{equation}
\label{decision2}
\overline{\alpha} = \max\{ \alpha(A(\delta)): \delta \in (1+\gamma)d^* {\bf \Delta}\}.
\end{equation}
If $\underline{\alpha} < 0$ and $\overline{\alpha} > 0$
then $d^*$ is certified by ZM with that confidence level $\gamma$.
This happens in all cases except 87, where ZM failed due to the large size.

We also compared $d^*$ to the result $d_{\rm F}$ of the technique \cite{fabrizi}, which is a sophisticated tool
tailored to problem (\ref{dist}). Column 6 of table \ref{table3} shows
perfect agreement on the bench from \cite{fabrizi}. Given the highly dedicated
character of \cite{fabrizi}, this can be understood as an endorsement
of our optimization-based approach.

\section*{Conclusion}
We have presented a bundle  trust-region  method for nonsmooth, nonconvex 
minimization, where cutting planes are tangents to a convex local model $\phi(\cdot,\x)$ of  $f$,
and where a trust-region strategy replaces the proximity control mechanism.
Global convergence of our method was proved
under natural hypotheses.

By way of an example we demonstrated that the standard approach in trust-region methods based on the Cauchy point fails
for nonsmooth functions.  We have identified a particular class $\mathscr S$  of nonsmooth functions, where the Cauchy
point argument can be salvaged. Functions in $\mathscr S$,
even when nonsmooth, can be minimized as if they were smooth.  The class $\mathscr S$  must therefore be regarded as
atypical in a nonsmooth optimization program,  and indeed, nonsmooth convex functions
are not in $\mathscr S$.

Algorithm \ref{algo1} was validated numerically on a test bench of 87 problems in automatic control,
where the versatility of algorithm \ref{algo1} with regard to the choice of the norm was  exploited. 
We were able to compute good quality lower bounds
for three NP-hard optimization problems related to the analysis of parametric robustness in system theory. 
In the majority of cases, posterior application of a global optimization technique allowed us to
certify these results as globally optimal.

\renewcommand\thesection{\arabic{section}}

\end{document}